\documentclass[aos]{imsart}

\RequirePackage{amsthm,amsfonts,amssymb}
\RequirePackage[numbers,sort]{natbib}
\RequirePackage[colorlinks,citecolor=blue,urlcolor=blue]{hyperref}
\RequirePackage{graphicx}

\usepackage[reqno]{amsmath}

\usepackage{enumerate}
\usepackage{dsfont}
\usepackage[usenames,dvipsnames]{color}
\usepackage{mathrsfs}
\usepackage{placeins}
\usepackage{graphicx}
\usepackage{tikz}
\usetikzlibrary{arrows}
\usepackage{url}
\usepackage{comment}
\usepackage{multicol}
\usepackage{cleveref}
\usepackage{stmaryrd}

\usepackage{natbib}

\renewcommand{\thetable}{\arabic{section}.\arabic{table}}
\numberwithin{equation}{section}

\theoremstyle{plain} 
\newtheorem{theorem}{Theorem}[section]
\newtheorem{lemma}[theorem]{Lemma}
\newtheorem{proposition}[theorem]{Proposition}
\newtheorem{corollary}[theorem]{Corollary}
\newtheorem{assumptionletter}{{{Assumption}}}

\theoremstyle{remark}
\newtheorem{definition}[theorem]{Definition}
\newtheorem{dfn}[theorem]{Definition}
\newtheorem{example}[theorem]{Example}

\newtheorem{remark}[theorem]{Remark}

\def\tablename{\footnotesize Table}

\newcommand{\bthe}{\begin{theorem}}
\newcommand{\ethe}{\end{theorem}}

\newcommand{\ben}{\begin{enumerate}}
\newcommand{\een}{\end{enumerate}}

\newcommand{\bit}{\begin{itemize}}
\newcommand{\eit}{\end{itemize}}

\newcommand{\beq}{\begin{equation}}
\newcommand{\eeq}{\end{equation}}

\newcommand{\ble}{\begin{lemma}}
\newcommand{\ele}{\end{lemma}}

\newcommand{\bde}{\begin{definition}\rm}
\newcommand{\ede}{\halmos\end{definition}}

\newcommand{\bco}{\begin{corollary}}
\newcommand{\eco}{\end{corollary}}

\newcommand{\bpr}{\begin{proposition}}
\newcommand{\epr}{\end{proposition}}

\newcommand{\brem}{\begin{remark}\rm}
\newcommand{\erem}{\end{remark}}

\newcommand{\bproof}{\begin{proof}}
\newcommand{\eproof}{\end{proof}}

\newcommand{\bexam}{\begin{example}\rm}
\newcommand{\eexam}{\end{example}}

\newcommand{\bfi}{\begin{fig}}
\newcommand{\efi}{\end{fig}}

\newcommand{\btab}{\begin{tab}}
\newcommand{\etab}{\end{tab}}

\newcommand{\beao}{\begin{eqnarray*}}
\newcommand{\eeao}{\end{eqnarray*}\noindent}

\newcommand{\balo}{\begin{align*}}
\newcommand{\ealo}{\end{align*}}

\newcommand{\balm}{\begin{align}}
\newcommand{\ealm}{\end{align}\noindent}

\newcommand{\beam}{\begin{eqnarray}}
\newcommand{\eeam}{\end{eqnarray}\noindent}

\newcommand{\barr}{\begin{array}}
\newcommand{\earr}{\end{array}}

\newcommand{\C}{\mathbb{C}}

\newcommand{\E}{\mathbb{E}}
\newcommand{\M}{\mathbb{M}}
\newcommand{\N}{\mathbb{N}}
\renewcommand\P{\mathbb{P}}
\newcommand{\R}{\mathbb{R}}
\newcommand{\MO}{\mathrm{MO}}

\newcommand{\VF}[1]{{\color{blue} #1}}

\def\MRV{\mathcal{MRV}}

\def\RV{\mathcal{RV}}

\def\bC{\mathbb{C}}
\def\bCA{\mathbb{CA}}

\def\bI{\mathbb{I}}
\def\bepsilon{\boldsymbol \varepsilon}

\def\cB{\mathcal{B}}
\def\cR{\mathcal{R}}
\def\e{\text{e}}

\def\bzero{\boldsymbol 0}
\def\bone{\boldsymbol 1}

\def\bB{\boldsymbol B}

\def\bR{\boldsymbol R}
\def\bS{\boldsymbol{S}}

\def\bX{\boldsymbol X}
\def\bY{\boldsymbol Y}
\def\bZ{\boldsymbol Z}
\def\bL{\boldsymbol L}
\def\bp{\boldsymbol p}
\def\bb{\boldsymbol b}

\def\bx{\boldsymbol x}
\def\by{\boldsymbol y}
\def\bz{\boldsymbol z}
\def\ba{\boldsymbol a}

\newcommand{\eqd}{\stackrel{\mathrm{d}}{=}}

\newcommand{\al}{{\alpha}}

\newcommand{\ov}{\overline}

\newcommand{\vague}{\stackrel{\lower0.2ex\hbox{$\scriptscriptstyle
                    \it{v} $}}{\rightarrow}}
\newcommand{\weak}{\stackrel{\lower0.2ex\hbox{$\scriptscriptstyle
                    \it{w} $}}{\rightarrow}}
\newcommand{\what}{\stackrel{\lower0.2ex\hbox{$\scriptscriptstyle
                    \it{\hat{w}} $}}{\rightarrow}}
\newcommand{\eqdis}{\stackrel{\lower0.2ex\hbox{$\scriptscriptstyle
                    \mathrm{d}$}}{=}}
\newcommand{\distr}{\stackrel{\lower0.2ex\hbox{$\scriptscriptstyle
                    \it{d} $}}{\rightarrow}}
\allowdisplaybreaks 
\definecolor{darkgreen}{RGB}{0,139,0}

\begin{document}

\begin{frontmatter}
\title{Aggregating heavy-tailed random vectors:\\ from finite sums to L\'evy processes}

\runtitle{Aggregating heavy-tailed vectors}

\begin{aug}
  \author[A]{\fnms{Bikramjit} \snm{Das}\ead[label=e1]{bikram@sutd.edu.sg}\orcid{0000-0002-6172-8228}}
    \and
   \author[B]{\fnms{Vicky} \snm{Fasen-Hartmann}\ead[label=e2]{vicky.fasen@kit.edu}\orcid{0000-0002-5758-1999}}
 \address[A]{Engineering Systems and Design, Singapore University of Technology and Design  \printead[presep={,\ }]{e1}}

   \address[B]{Institute of Stochastics, Karlsruhe Institute of Technology\printead[presep={,\ }]{e2}}


  \runauthor{B. Das and  V. Fasen-Hartmann}
\end{aug}

\begin{abstract}
The tail behavior of aggregates of heavy-tailed random vectors is known to be determined by the so-called principle of ``one large jump'', be it for finite sums, random sums, or, L\'evy processes. We establish that, in fact, a more general principle is at play. Assuming that the random vectors are  {multivariate} regularly   varying on various subcones of the positive quadrant, first we show that their  aggregates are also {multivariate} regularly varying on these subcones. This allows us to approximate certain tail probabilities  which were rendered asymptotically negligible under classical regular variation, despite the  ``one large jump'' asymptotics.  We also discover that depending on the structure of the tail event of concern, the tail behavior of the aggregates may be characterized by {more than a single large jump}.  Eventually, we illustrate a similar phenomenon for multivariate regularly varying L\'evy processes, establishing as well a relationship between   multivariate regular variation of a L\'evy process and multivariate regular variation of its L\'evy measure on different subcones.

\end{abstract}



\end{frontmatter}

\section{Introduction}\label{sec:intro}

In this paper we study the behavior of the asymptotic tail distribution of independent sums of heavy-tailed random vectors under the paradigm of multivariate regular variation \citep{bingham:goldie:teugels:1989, resnickbook:2007}. Assessment of such tail probabilities  are of interest in risk management for many finance, insurance, queueing, and environmental applications \citep{McNeil:Frey:Embrechts,embrechts:kluppelberg:mikosch:1997,asmussen:2003}. 
Multi-dimensional tail events are often characterized by at least one variable exceeding a high threshold, and the asymptotic probability of such events follow the so-called ``one large jump'' principle, see \cite{hult:lindskog:2006b}.

Assume that all our random elements are defined on the same probability space $(\Omega,\mathcal{A},\P)$. If $Z,Z^{(1)},\ldots, Z^{(n)}$ are independent and identically distributed (i.i.d.) random variables, then  for fixed $x>0$ we know that
\begin{align}\label{eq:unisub}
 \P(Z^{(1)}+\ldots+Z^{(n)}>tx) \sim n\,\P(Z>tx) 
 \quad \text{ as }t\to\infty,
 \end{align}
  if and only if $Z$ is subexponential, i.e.,  $\P(Z^{(1)}+Z^{(2)}>t) \sim 2\,\P(Z>t)$ as $t\to \infty$; \citet{chistyakov:1964} proved this for non-negative random variables, later extended to $\R$ by \citet{willekens:1986}. This phenomenon  of ``one large jump'' exhibited in \eqref{eq:unisub} is so called because a high threshold-crossing of the sum of a set of random variables occurs with the same asymptotic probability as any one of them crossing the same threshold.
  Recall that a random variable $Z$ has a regularly varying right tail if for $x>0$, we have $\lim_{t\to\infty} \P(Z>tx)/\P(Z>t) = x^{-\alpha}$ for some  $-\alpha<0$, which is called the index of regular variation or tail index. We write $Z\in\RV_{-\alpha}$. Regularly varying distributions are subexponential as well (cf. \cite{teugels:1975}), i.e, and hence \eqref{eq:unisub} holds when $Z\in\RV_{-\alpha}$.

In a multivariate context, an $\R_+^d$-valued random vector $\bZ$ is {multivariate regularly varying (MRV) on $\R_+^d\setminus\{\bzero\}$, if there exists a function $b(t)\to \infty$ as $t\to \infty$, and a non-null measure $\mu$ in the set of all Borel measures in $\R_+^d\setminus\{\bzero\}$ which are finite on sets bounded away from $\bzero$ such that
\[t\,\P(\bZ/b(t) \in A) \to \mu(A) \quad \text{ as } t\to\infty \]
for Borel sets $A$ bounded away from $\bzero$ with $\mu(\partial A)=0$; see Section \ref{sec:prelim} for details. In particular, we have $\mu(tA)=t^{-\alpha}\mu(A)$ for some $\alpha>0$ and write $\bZ\in \MRV(\alpha, b, \mu, \R_+^d\setminus\{\bzero\})$. Now, for i.i.d. non-negative random vectors $\bZ,\bZ^{(1)}, \ldots, \bZ^{(n)}$  with $\bZ \in \MRV(\alpha, b, \mu, \R_+^d\setminus\{\bzero\})$, we can deduce that
\begin{eqnarray} \label{eq:3}
    \bZ^{(1)} + \ldots + \bZ^{(n)} \in \MRV(\alpha, b, n \mu, \R_+^d\setminus\{\bzero\})
\end{eqnarray}
(cf. {\cite[Proposition 4.1]{resnick:1986}}, \cite[Section 7.3]{resnickbook:2007} and \cite[Lemma 3.11]{jessen:mikosch:2006}). Hence, for Borel sets $A$ bounded away from $\bzero$ with $\mu(\partial A)=0$ and fixed $n\ge 1$  we may approximate
\begin{align}\label{eq:nMRV}
\P(\bZ^{(1)} + \ldots + \bZ^{(n)} \in tA) \sim  \frac{n}{b^{\leftarrow}(t)}\,\mu(A) \sim n \,\P(\bZ\in tA)
\quad \text{ as } t\to\infty,
\end{align}
 if $\mu(A)>0$. Thus, \eqref{eq:nMRV} extends \eqref{eq:unisub} to higher dimensions and the principle of ``one large jump'' appears to hold.  Curiously though, one often encounters examples where for a large class of sets $A$, the value of $\mu(A)$ is equal to zero and in this case,
 \begin{align}\label{eq:nMRV+}
    \lim_{t\to\infty}b^{\leftarrow}(t)\,\P(\bZ^{(1)} + \ldots + \bZ^{(n)} \in tA)=0=  \lim_{t\to\infty}b^{\leftarrow}(t)\,\P(\bZ\in tA).
\end{align}
This makes \eqref{eq:nMRV} of limited practical use. For example, if the elements of $\bZ=(Z_1,\ldots,Z_d)$ are themselves i.i.d. (with regularly varying marginal tail distributions) and
  we consider sets
\begin{align}\label{eq:setA}
A=\{\bz\in \R_+^d: z_j > x_j \; \forall\; j \in S\},
\end{align} for indices $S\subseteq \mathbb{I}:=\{1,\ldots,d\}$ with $|S|\ge 2$, $x_j>0$ for $j\in S$, then for $\bZ\in tA$ to hold, at least two components of $\bZ$ need to be large together and under classical MRV assumptions we have $\mu(A)=0$ verifying \eqref{eq:nMRV+} (cf. \Cref{ex:cop:ind}). In fact, the components of $\bZ$ need not to be independent at all, a Gaussian dependence among variables with tail equivalent regularly varying marginal distributions and pairwise correlations less than one will ensure $\mu(A)=0$;  see \Cref{subsec:copasymind} for further examples.

In such a case, we do observe that a notion subtler than the well-known ``one large jump'' phenomenon holds { depending on the joint dependence of the underlying random vector and on the type of tail set $A$ considered.} If multivariate regular variation  holds on relevant subcones of $\R_+^d$ we realize that in many scenarios,
\begin{align}\label{eq:nMRVgeneral}
\P(\bZ^{(1)} + \ldots + \bZ^{(n)} \in tA)\sim { \frac{n}{\widetilde b^{\leftarrow}(t)}\,\widetilde \mu(A)}   \sim n \,\P(\bZ\in tA)
\quad \text{ as } t\to\infty,
\end{align}
{where $0<\widetilde \mu(A)<\infty$ and $\widetilde b^{\leftarrow}(t)$ is a function satisfying $\lim_{t\to\infty}
 b^{\leftarrow}(t)/\widetilde b^{\leftarrow}(t)=0$.}
This means that the left and the right hand side of \eqref{eq:nMRV} 
are asymptotically equivalent, even though the term in the middle contains $\mu(A)=0$. 
The relevant results and examples where we observe \eqref{eq:nMRVgeneral} are detailed in \Cref{subsec:hrvinallcones}. In particular, the results hold for more general sets than the ones  given in \eqref{eq:setA}, and  we can specify $\widetilde \mu(A)$ and $\widetilde b^{\leftarrow}(t)$. Formally, we show that $\sum_{k=1}^n\bZ^{(k)}$ is multivariate regularly varying on subcones of $\R_+^d$,
obeying \eqref{eq:3} but with $\alpha, b, \mu$ replaced by $\widetilde{\alpha}, \widetilde{b}, \widetilde{\mu}$ and $\R_+^d\setminus\{\bzero\}$  replaced by an appropriate subcone.
Note that \eqref{eq:nMRVgeneral} hints at connections to the notion of multivariate subexponentiality  (cf.  \cite{Cline:Resnick:1992,Omey:2006,samorodnitsky:sun:2016}).

Another phenomenon  investigated  for sets $A$ as defined in \eqref{eq:setA} is that the tail event $\bZ^{(1)}+\ldots+\bZ^{(n)} \in tA$ may be determined by threshold crossings in different co-ordinates of $S$, by different variables $\bZ^{(k)}$. Thus, an aggregation of random vectors leads to a tail event with a ``few large jumps'' and these jumps occur either together in one random vector $\bZ^{(k)}$, or separately in a few different vectors; and  for $A$ as defined in \eqref{eq:setA} with $|S|=i$, we get
\begin{align*}
  (b^{\leftarrow}(t))^i \,\P(\bZ^{(1)} + \ldots + \bZ^{(n)} \in tA) & \sim C_{n,i} \prod_{j\in S} x_j^{-\alpha} \sim C_{n,i} \prod_{j\in S} b^{\leftarrow}(t) \,\P(Z_j > tx_j)  
\end{align*}
as $t\to\infty$ (cf. \Cref{rem:delta1}). Thus,
\begin{align}\label{eq:nMRVgeneral2}
\P(\bZ^{(1)} + \ldots + \bZ^{(n)} \in tA)  \sim C_{n,i} \prod_{j\in S}\,\P(Z_{j}\in tx_j) \quad \text{ as } t\to\infty,
\end{align}
where $C_{n,i}>0$ is not necessarily equal to $n$, and depends on the number of summands $n$, the index $i=|S|$ which represents the type of set $A$, as well as the distribution of $\bZ$; the associated results are discussed  in \Cref{prop:ind} (\Cref{subsec:hrvinallcones}) and \Cref{subsec:nohrvinallcone}. We also notice that this phenomenon is often observed under the more general assumption of adapted-MRV (cf. \Cref{def:amrv}) for the underlying random vectors (cf. \Cref{subsec:nohrvinallcone}}). {Since $C_{n,i}\not=n$ in this example, it is no surprise that the limit measure observed in aggregating adapted-MRV is not linear in $n$ anymore, in contrast to \eqref{eq:3}.}

One of our primary interests  is to characterize the asymptotic behavior of multi-dimensional regularly varying L\'evy processes which have inherent applications to stochastic storage processes including insurance claims, inventory management, and more (cf. \cite{prabhu:1998,asmussen:2003}). This also happens to be  a natural progression from computing tail probabilities of finite aggregation of random vectors. A L\'evy process $\bL=(\bL(s))_{s\geq 0}$,
is a stochastic process    with $\bL(0)=\bzero$ $\mathbb{P}$-almost surely, has stationary and independent increments, and has c\`adl\`ag sample paths (cf. \cite{sato:1999}). Consequently, $\bL(s)$ is infinitely divisible and has
the same distribution as sums of i.i.d. random vectors; following the basic premise of this paper. A L\'evy process $\bL$ is characterized by its L\'evy measure $\Pi(A)$, which measures the expected
number of jumps of $\bL$ in $[0,1]$ whose jump sizes are in $A$ (cf. \Cref{section:Levy}).
The principle of one large jump is illustrated for multivariate regularly varying L\'evy processes by \citet{hult:lindskog:2005SPA,hult:lindskog:2006b} and the asymptotic rates of further hidden jumps have been characterized by  \citet{lindskog:resnick:roy:2014} (for the univariate case).
Our work addresses the case where the results of \cite{hult:lindskog:2006b} hold, including and specifically addressing cases with negligible probability approximation for the tail event.
{In particular,} a conclusion of \cite[Proposition 3.1]{hult:lindskog:2006b}  is that
$\bL(1)\in\MRV(\alpha, b, \mu, \R_+^d\backslash\{\bzero\})$ if and only if
 $\Pi
\in\MRV(\alpha, b, \mu, \R_+^d\backslash\{\bzero\})$ and then, for any Borel set $A$ bounded away from $\bzero$ with $\mu(\partial A)=0$  we have
\begin{eqnarray} \label{eq1}
    \P(\bL(s)\in tA) \sim  \frac{s}{b^{\leftarrow}(t)}\,\mu(A)\sim  s\,\Pi(tA) \quad \text{ as } t\to\infty,
\end{eqnarray}
if $\mu(A)>0$ (see \cite{embrechts:goldie:veraverbeke:1979} for the one dimensional case). Naturally, if $\mu(A)=0$ then \eqref{eq1} gives a zero estimate. For example, this  happens if we consider $\bL$ to be comprised of $d$ i.i.d. one-dimensional regularly varying L\'evy processes and $A$ is defined as in \eqref{eq:setA}.  Under quite
general conditions, if the L\'evy measure is multivariate regular varying on relevant subcones of $\R_+^d$, we
show  $\bL(s)$ is multivariate regularly varying on that subcone as well.
As a consequence under these conditions, we observe that,
\begin{eqnarray} \label{eq2}
    \P(\bL(s)\in tA)  \sim s\,\P(\bL(1)\in tA) \sim  \frac{s}{\widetilde b^{\leftarrow}(t)}\,\widetilde \mu(A)\sim  s\,\Pi(tA)  \quad \text{ as } t\to\infty,
\end{eqnarray}
for  sets $A$ in appropriate subcones; here the function $\widetilde{b}^{\leftarrow}(t)$ satisfies $\lim_{t\to\infty}
 b^{\leftarrow}(t)/\widetilde b^{\leftarrow}(t)=0$ and in contrast to \eqref{eq1} where $\mu(A)=0$, here we have $0<\widetilde{\mu}(A)<\infty$. However,  we also find cases where the asymptotics are different and we observe
 \begin{eqnarray} \label{eq3}
    \P(\bL(s)\in tA) \sim \phi(s)\,\Pi(tA) \quad \text{ as } t\to\infty,
\end{eqnarray}
where $\phi$ is a function for which $\phi(s)=s$ may not hold, defying the linearity property often observed for L\'evy processes, cf. \eqref{eq2}.
These results on the tail probabilities of heavy-tailed L\'evy processes have obvious implications on risk and ruin problems, especially in the context of insurance and finance, and is addressed in an associated article \cite{das:fasen:2023}.


This paper is organized as follows. In \Cref{sec:prelim}, we provide necessary preliminary results and background for our work; first we discuss the basic framework in terms of multivariate regular variation on subcones of $\R_+^d$ using $\M$-convergence.  We also discuss copulas and survival copulas used to model dependence in the said random vectors.
Our main result, \Cref{mainTheorem}, appears in \Cref{sec:mainresult} which is used to obtain results on ``{one or few large jumps}'' of the form \eqref{eq:nMRVgeneral} and \eqref{eq:nMRVgeneral2} under a variety of assumptions. In particular, we show
how  multivariate regular variation  of two independent random vectors on a subcone of $\R_+^d$ is used to obtain multivariate regular variation of their $n$-convolution. Our results are complemented with examples where they can be applied, especially for the convolution of finitely many i.i.d. random vectors. In Section \ref{sec:randomsum}, the results on finite convolutions are extended to random sums and finally applied to assess the tail behavior in regularly varying L\'evy processes in Section \ref{section:Levy}. The  proofs in the paper are relegated to the appendices.

\section{Preliminaries}\label{sec:prelim}
In this section, we discuss the framework for assessing probabilities of tail events where joint thresholds may be crossed; we briefly recall the theory of $\M$-convergence used to define multivariate regular variation on subcones of $\R_+^d=[0,\infty)^d$ (cf. \citep{lindskog:resnick:roy:2014,das:fasen:kluppelberg:2022}). Furthermore, we characterize convergence for different types of tail sets in \Cref{prop:rectsetsforM}. The notion of multivariate regular variation on different cones of $\R_+^d$ is extended in \Cref{def:amrv} to allow for a broader class of models and examples. This allows for a framework where the sum of two random vectors can be MRV on a specific subcone even if neither of the two summands is MRV on that subcone. In \Cref{subsec:copasymind}, we provide examples of joint distributions where our framework is useful, we use copulas to represent joint dependence.  This allows us to illustrate our results in a variety of examples.


Unless otherwise specified,  all random vectors are assumed to lie on the positive quadrant $\E_d:=\R_{+}^{d}$. Notationally, vector operations are  understood component-wise, e.g., for vectors $\bz=(z_1,\ldots,z_d)$ and $\bx=(x_1,\ldots,x_d)$, $\bx\le \bz$ means $x_i\le z_i$ for all $i$.  Moreover, for a constant $t>0$ and a set $A\subseteq \R_+^{d}$, we denote by {$tA:= \{t\bz: \bz\in A\}$.}
\subsection{Regular variation} \label{subsec:regvar}
The theory of regular variation provides a systematic framework to discuss heavy-tailed distributions; see \citet{bingham:goldie:teugels:1989,resnickbook:2007} for details. Here we briefly discuss regular varying functions and multivariate regular variation of measures and random vectors on Euclidean cones with $\M$-convergence.

  A function $f: \R_{+} \to \R_{+}$ is \emph{regularly varying} at infinity if for all $x>0$, we have $\lim_{t\to\infty} f(tx)/f(t) =x^{\beta}$
for some fixed $\beta\in \R$. We write $f\in\RV_{\beta}$; and if $\beta=0$, then the  function $f$ is called \emph{slowly varying}.  A real-valued random variable $Z $ with distribution function $F$, denoted by $Z\sim F$,  is regularly varying (at $+\infty$) if $\overline{F}:= 1-F \in \RV_{-\alpha}$ for some $\alpha> 0$. Equivalently, $Z\sim F$ is regularly varying with index $-\alpha< 0$ if there exists a measurable function $b:\R_+\to\R_+$ with $b(t)\to \infty$ as $t\to\infty$ such that
\[t\,\P(Z>b(t)x) = t\,\overline{F}(b(t)x) \xrightarrow{t \to \infty} x^{-\alpha} \quad \forall\, x>0.\]
We write $\overline{F}\in \RV_{-\alpha}(b)$. As a consequence $b(t)\in \RV_{1/\alpha}$ and a canonical choice for $b$ is
$b(t)= F^{\leftarrow}(1-1/t)=\overline{F}^{\leftarrow}(1/t)$ where $F^{\leftarrow}(x)= \inf\{y\in\R:F(y)\ge x\}$.

Measures $\mu$ and $\nu$ defined on $\mathcal{B}(\R_+)$ are \emph{(right) tail equivalent} if
\begin{align}\label{eq:measuretailequiv}
    \lim_{t\to\infty}\frac{\mu((t,\infty))}{\nu((t,\infty))}  = c,
\end{align}
for some $c>0$.
Naturally, the same holds for probability measures and hence,
distribution functions $F$ and $G$ are \emph{(right) tail equivalent} if
\begin{align}\label{eq:tailequiv}
    \lim_{t\to\infty}\frac{\overline{F}(t)}{\overline{G}(t)} = \lim_{t\to\infty} \frac{1-{F}(t)}{1-{G}(t)} = c,
\end{align}
for some $c>0$.  We call measures $\mu$ and $\nu$ (respectively distributions $F$ and $G$) \emph{completely} tail equivalent if \eqref{eq:measuretailequiv} (respectively \eqref{eq:tailequiv}) holds with $c=1$. We often assume that components of the random vectors considered in this paper are tail equivalent (if not identically distributed, or completely tail equivalent).

We discuss (multivariate) regular variation  on Euclidean subspaces of $\E_d=\R_{+}^{d}$ using $\M$-convergence of measures  which differs from vague convergence, the traditional notion used for multivariate regular variation. See \cite{hult:lindskog:2006a,lindskog:resnick:roy:2014,das:fasen:kluppelberg:2022} for further details and the preference for this notion over vague convergence; moreover see \cite{basrak:planinic:2019} for a broader notion of vague convergence.

Consider the space $\E_d$ endowed with the sup-norm metric $d(\bx,\by)=||\bx-\by||_{\infty}$. A \emph{cone} $\mathbb{C} \subset \E_d $ is a set which is closed under scalar multiplication: if $\bz \in
\mathbb{C}$ then $t\bz \in \mathbb{C}$ for $t>0$; a \emph{closed cone}  is a cone which is a closed set in $\E_d$.
Regular variation is defined using $\M$-convergence on a closed
cone $\mathbb{C}\subset \mathbb{E}_d$ with a closed cone  $\mathbb{C}_0 \subset  \mathbb{C}$ deleted.  We say that a  subset $A \subset \mathbb{C} \setminus \mathbb{C}_0$ is {\it bounded away from\/} $\mathbb{C}_0$ if $d(A,\mathbb{C}_0)=\inf\{d(\bx,\by):\bx\in A,\by\in\C_0 \}>0$. Denote by $\M(\mathbb{C}\setminus \mathbb{C}_0)$, the class of Borel measures on $\C\setminus\C_0$ assigning finite measures to all Borel sets $A\subset \C\setminus\C_0$, which are bounded away from $\C_0$. {We often refer to a subspace of the closed cone $\E_d$ which is a cone, as a \emph{subcone}.}

Let $\C_0\subset \C\subset \R_+^d$ be closed cones containing $\bzero$. We  define $\M$-convergence first and subsequently use it to define regular variation on $\M(\mathbb{C}\setminus \mathbb{C}_0)$.
\begin{dfn}[$\M$-convergence]\label{def:mconv}
Let $\mu,\mu_n, n\ge 1$ be Borel measures on $\M(\mathbb{C}\setminus \mathbb{C}_0)$. Suppose $\int f\,d\mu_n\to\int f\,d\mu$ as $n\to\infty$ for any bounded, continuous, real-valued function $f$ whose support is bounded away from $\C_0$, then we say {\it $\mu_n$ converges to $\mu$ in $\M(\mathbb{C}\setminus \mathbb{C}_0)$}, and write $\mu_n\to\mu$ in $\M(\C\backslash\C_0)$.
\end{dfn}
 Next we define regular variation of measures on $\M(\bC\setminus\bC_0)$ which is an extension of the definition found in \cite{hult:lindskog:2006a} for measures in $\M(\R_+^d\setminus \{\bzero\})$.

 
\begin{dfn}[Regular variation of measures]\label{def:mrvmeasure}
A Borel measure $\Pi$ on $\M(\mathbb{C}\setminus \mathbb{C}_0)$ is regularly varying on $\C\setminus\C_0$ if there exists
a regularly varying function $b(t) \in \RV_{1/\alpha}$,  $\alpha>0$, called the {\it scaling
  function\/} and a non-null (Borel) measure $\mu(\cdot)
\in \M(\mathbb{C}\setminus \mathbb{C}_0)$ called the {\it limit or tail measure\/} such that as $t \to\infty$,
\begin{equation}\label{eq:RegVarMeasure}
t\,\Pi( b(t) \,\cdot \,) \to \mu(\cdot),
\end{equation}
in $\M(\mathbb{C}
\setminus \mathbb{C}_0)$. 
 We write $\Pi \in \MRV(\alpha, b, \mu, \mathbb{C}
\setminus \mathbb{C}_0)$.
\end{dfn}
The next definition is a modification of \Cref{def:mrvmeasure} for multivariate probability measures and hence, random vectors in $\R^d_+$.
\begin{dfn}[Multivariate regular variation]\label{def:mrv}
A random vector $\bZ \in \R_{+}^{d}$ is multivariate regularly varying on $\mathbb{C}
\setminus \mathbb{C}_0$ if there exists a regularly varying function
$b(t) \in \RV_{1/\alpha}$, $\alpha >0$,  and a non-null (Borel) measure $\mu(\cdot)
\in \M(\mathbb{C}
\setminus \mathbb{C}_0)$ 
{such that} as $t \to\infty$,
\begin{equation}\label{eq:RegVarMeas}
t\,\P\left( \bZ/b(t) \in \,\cdot \,\right) \to \mu(\cdot),
\end{equation}
in $\M(\mathbb{C}
\setminus \mathbb{C}_0)$. We write $\bZ \in \MRV(\alpha, b, \mu, \mathbb{C}
\setminus \mathbb{C}_0)$ and one or more parameters are often dropped according to convenience.
\end{dfn}

\begin{remark}
In both Definitions \ref{def:mrvmeasure} and \ref{def:mrv},  since $b(t) \in \RV_{1/\alpha}$,
we observe that the limit measure $\mu(\cdot) $ has the scaling property 
$\mu(t \;\cdot\,) =t^{-\alpha} \mu(\,\cdot\,)$ for $t>0$. Hence, if the measure or the random vector is $\MRV(\alpha,b, \mu, \mathbb{C}
\setminus \mathbb{C}_0)$, we often refer to $-\alpha<0$ as its tail index (in the subspace $\mathbb{C}
\setminus \mathbb{C}_0$). 
\end{remark}
\subsection{Regular variation on co-ordinate subcones of the positive quadrant}\label{subsec:hrv}
Equipped with the notion of $\M$-convergence and regular variation,  we proceed to discuss regular variation on a particular set of subcones of $\R_+^d$ and also provide equivalent conditions for the same  (cf. \cite{mitra:resnick:2011hrv}, \cite[Section 2]{das:fasen:kluppelberg:2022}). For $\bz\in \R^{d}_+$ write $\bz= (z_{1}, \ldots,z_{d})$ and denote the (decreasing) order statistics  of $\bz$ by $z_{(1)}\ge z_{(2)}\ge \ldots \ge z_{(d)}$. For $0\le i \le d-1$ let
\begin{align*}
    \bCA_d^{(i)} & := \bigcup_{1\le j_{1}<\ldots<j_{d-i} \le d}  \{\bz\in \R^{d}_+: z_{j_{1}}=0,\ldots, z_{j_{d-i}}=0\}
                 = \{\bz\in \R^{d}_+: z_{(i+1)}=0\}, \quad
\end{align*}
and  $\bCA_d^{(d)}:=\{\bz\in \R^d_+: z_{(d)}>0\}$. For any $i=1,\ldots, d$, the closed cone $\bCA^{(i)}_{d}$ represents the union of all $i$-dimensional co-ordinate hyperplanes in $\R^{d}_{+}$.
Define the following sequence of subcones of $\R^d_+$ where we investigate regular variation (when it exists):
{
\begin{align}
\E^{(i)}_{d}  := &\;  \R^d_+ \setminus \bCA^{(i-1)}_{d} = \; \{\bz\in \R^d_+: z_{(i)}>0\}, \quad\quad 1\le i \le d. \label{def:edd}
\end{align}}
 We call the subsets $\E_d^{(i)}$ \emph{co-ordinate subcones} since they are cones obtained from $\R_+^d$ by removing particular co-ordinate hyperplanes. Here $\E^{(1)}_{d}$ is the {positive quadrant} with $\{\bzero\}=\bCA^{(0)}_{d}$ removed, $\E^{(2)}_{d}$ is the {positive quadrant} with all one-dimensional co-ordinate axes removed, $\E^{(3)}_{d}$ is the {positive quadrant} with all two-dimensional co-ordinate hyperplanes removed, and so on.  Clearly,
$$\E^{(1)}_{d} \supset \E^{(2)}_{d} \supset \ldots \supset \E^{(d)}_{d}=\bCA_d^{(d)}.$$
\begin{remark}[Asymptotic tail independence and hidden regular variation]
Suppose a random vector $\bZ\in \R^d_+$ admits regular variation on $\E_d^{(i)}$ with $\bZ\in \MRV(\alpha_i,b_i,\mu_i,\E_d^{(i)})$ and $\mu_i(\E_d^{(i+1)})=0$. We interpret this to be \emph{asymptotic tail independence at level $i$}, meaning the probability with  which $i+1$ or more components are simultaneously large, is negligible in comparison to the probability with  which $i$ (or fewer) components are simultaneously large. Suppose for some $j>i\ge 1$ we have $\bZ\in \MRV(\alpha_j,b_j,\mu_j,\E_d^{(j)})$ with $\lim_{t\to\infty} ({b_i(t)}/{b_j(t)}) =\infty$ then we say $\bZ$ exhibits  \emph{hidden regular variation} (HRV) on $\E_d^{(j)}$ with respect to MRV on $\E_d^{(i)}$.
\end{remark}

 We investigate regular variation on cones of the form $\E_d^{(i)}$, since joint exceedances often occur in tail subsets of such cones. A recipe for seeking (hidden) regular variation on such subsets have been discussed in Section 2.1 of \cite{das:fasen:kluppelberg:2022} and we do not repeat the steps here. In the rest of this section, we characterize a particular family of sets $\cR^{(i)}$ (defined in \eqref{eq:Ri} below), proving {that it is an $\M$-convergence determining class on $\E_d^{(i)}$}.  The particular tail sets appear in multivariate risk and reliability problems where the quantity of interest is a finite or a random sum of identically distributed vectors.

 Let $\cB:=\cB(\R_+^d)$ denote the Borel $\sigma$-algebra on $\R_+^d$. For any $i\in \{1,\ldots, d\}=:\mathbb{I}$,  $\E_d^{(i)}$ is a subspace of $\R_+^d$ and we denote its induced $\sigma$-algebra  by
\begin{align}\label{eq:Bi}
\cB^{(i)}: = \cB(\E_d^{(i)}) = \{A\in \cB: A\subseteq \E_d^{(i)} \}.
\end{align}
A \emph{rectangular set} in $\E_d^{(i)}$ is defined as any set $A= \{\bz \in \R_+^d: z_j>x_j\;\forall\, j\in S\}$ where $S\subseteq \mathbb{I}$, $|S|\ge i$ and $x_j>0\; \forall\, j\in S$. Let us denote the collection
\begin{align}\label{eq:Ri}
\cR^{(i)}:= \{A\in \cB^{(i)}: A \; \text{is a rectangular set in } \E_d^{(i)} \}.
\end{align}

\begin{lemma}\label{lem:RiisPi}
$\cR^{(i)}$ is a $\pi$-system and  $\sigma(\cR^{(i)})=\cB^{(i)}$.
\end{lemma}
The proofs of \Cref{lem:RiisPi} as well as other subsequent results in this section are given in \Cref{sec:proofprelim}. 
The following proposition shows that for verifying convergence of measures in $\E_d^{(i)}$, we can restrict to testing convergence in sets belonging to $\cR^{(i)}$. The result and the proof are in the spirit of \cite[Lemma 6.1]{resnickbook:2007}. 
\begin{proposition} \label{prop:rectsetsforM}
Let $\mu, \mu_t \in \M(\E_d^{(i)})$ for all $t>0$ and some fixed $i\in \bI$. Then  as $t\to\infty$,
\begin{align}\label{Mcon:mut}
\mu_t \to \mu \quad \text{in} \quad \M(\E_d^{(i)})
\end{align}
 if and only if
 \begin{align}\label{Mcon:mutAy}
 \lim_{t\to\infty}\mu_t(A) \to \mu(A), \quad \forall\; A\in \cR^{(i)}
 \end{align}
with $\mu(\partial A)=0$ ($\mu$-continuity set), where $\cR^{(i)}$ is the collection of sets as defined in \eqref{eq:Ri}.
\end{proposition}

\begin{remark}
In light of Proposition~\ref{prop:rectsetsforM}, when we seek regular variation (or any measure convergence) in the space $\E_d^{(i)}$ using $\M$-convergence (as in Definition \ref{def:mrv}), we can equivalently show this only for rectangular sets in $\cR^{(i)}$ (which are also continuity sets with respect to the limit measure).
\end{remark}
We wrap this section up with  an extension of the so-called ``heavier tail wins" phenomenon, in the context of multivariate regular variation on a subcone of $\R_+^d$. It is useful for many of the proofs in this paper.

\begin{lemma} \label{Lemma_aux}
Suppose $\bX\in \MRV(\alpha_i,b_i,\mu_i,\E_d^{(i)})$
for some fixed $i\in\mathbb{I}$, and $\bX$ is independent of the $\R^d$-valued
random vector $\bY$ with $\E\|\bY\|^{\alpha_i+\gamma}<\infty$
for some $\gamma>0$. Then
\begin{eqnarray*}
    \bX+\bY\in\MRV(\alpha_i,b_i,\mu_i,\E_d^{(i)}).
\end{eqnarray*}
\end{lemma}

\subsection{A joint multivariate regular variation condition}\label{subsec:jointmrv}
Since our interest is in the tail behavior of aggregates over independent regularly varying vectors, when considering joint exceedances of sums of such vectors, regular variation on various combinations of subcones become important. Incidentally, in certain scenarios we encounter a sequence of random vectors, none of which possess MRV on a particular subcone, yet their sum admits MRV on the same subcone. The following definition provides conditions for such  vectors to be tractable under aggregation in the multivariate regular variation framework.

\begin{definition}[Adapted Multivariate Regular Variation]\label{def:amrv}
Suppose $\bZ\in \R_+^d$ is a random vector such that the following holds:
\begin{enumerate}
    \item $\bZ \in \MRV(\alpha_i,b_i,\mu_i,\E_d^{(i)})$  for $i=1,\ldots, \Delta$ for some $\Delta \le d$.
    \item If $\Delta<d$, then additionally assume that there exists a $\gamma\in (0,\alpha_1/d)$, such that for \linebreak $i=\Delta+1,\ldots, d$ and any $A\in \cR^{(i)}$, we have
\begin{align}\label{cond:nullcon}
 \lim_{t\to\infty} t\,\P\left(\frac{\bZ}{b_i(t)}\in \, A\right)  =0
\end{align}
    where $b_i(t):=t^{1/(i(\alpha_1+\gamma))}$, i.e., with the rate $b_i$, we have convergence to zero. We refer to \eqref{cond:nullcon} as \emph{null convergence} and write $\bZ\in \mathcal{NC}(b_i(t),\E_d^{(i)})$.
\end{enumerate}
Then we say $\bZ$ is \emph{adapted multivariate regular varying} (or, \emph{adapted-MRV}) on $\R^d_{+}$ and write
$\{\bZ\in \MRV^*(\alpha_i,b_i,\mu_i,\E_d^{(i)}); i=1,\ldots, d; \Delta\}$ where $\alpha_i=\infty, \mu_i\equiv 0$ for $i=\Delta+1, \ldots, d$.
\end{definition}
Adapted-MRV is defined jointly for all cones $\E_d^{(i)}, i=1,\ldots, d$. Note the following.
\begin{enumerate}
\item[i)] For $i=1,\ldots, \Delta$, $\MRV^*(\alpha_i,b_i,\mu_i,\E_d^{(i)})$ and $\MRV(\alpha_i,b_i,\mu_i,\E_d^{(i)})$ are equivalent.
\item[ii)] For  $i=\Delta+1,\ldots,d$, the notation $\MRV^*(\alpha_i,b_i,\mu_i,\E_d^{(i)})$ means that for some \linebreak $0<\gamma<\alpha_1/d$, we have  $\bZ\in \mathcal{NC}(b_i(t)=t^{1/(i(\alpha_1+\gamma))}, \E_d^{(i)})$ and $\alpha_i=\infty, \mu_i\equiv 0$. The constant $\gamma$ is chosen to be the same for all $i=\Delta+1,\ldots,d$ and the value of  $\Delta=\arg \max_i \{\alpha_i: \alpha_i<\infty\}$ is implicit.
\end{enumerate}

\begin{example}\label{ex:vecnomrvsummrv2}
  Let $X^{(1)}, X^{(2)} \sim F$ be i.i.d  random variables with $\P(X^{(1)}>x)=x^{-\alpha}, x>1$ for some $\alpha>0$.  Let $B^{(1)}, B^{(2)}$ be i.i.d  random variables with $\P(B^{(1)}=1)=0.5=\P(B^{(1)}=0)$. Define for $k=1,2$,
  \[\bZ^{(k)}: = B^{(k)}(X^{(k)},0) + (1-B^{(k)})(0,X^{(k)}).\]
  Then  $\bZ^{(1)}, \bZ^{(2)}$ are i.i.d. with $\bZ^{(1)}\in \MRV(\alpha, b(t)=t^{1/\alpha}, \mu_1, \E_2^{(1)})$ where
  \[\mu_1(([0,x_1]\times[0,x_2])^c) = 0.5 x_1^{-\alpha} + 0.5 x_2^{-\alpha}, \quad x_1, x_2 >0.\]
  Clearly, $\bZ^{(1)} \in \mathcal{NC}(b^*(t)=t^{1/2(\alpha+\gamma)}, \E_2^{(2)})$ for any $\gamma>0$. Hence, $\bZ^{(k)}$ is adapted-MRV with $\Delta=1$. But we can check that  $\bZ^{(1)}+\bZ^{(2)}\in \MRV(\alpha, b(t)=t^{1/\alpha}, 2\mu_1, \E_2^{(1)})$  and $\bZ_1+\bZ_2\in \MRV(2\alpha, b(t)=t^{1/(2\alpha)}, \mu_2, \E_2^{(2)})$  where
  \[\mu_2((x_1,\infty)\times(x_2,\infty)) = 0.5 (x_1x_2)^{-\alpha}, \quad x_1, x_2 >0.\]
\end{example}

\begin{remark} Many examples of regularly varying vectors on $\E_d^{(1)}$ are in fact adapted-MRV, including all examples mentioned in \Cref{subsec:copasymind} and more. A few comments on how \eqref{cond:nullcon} enriches our class of models follows.
\begin{enumerate}
  \item[(a)] If $\bZ\in \MRV(\alpha_i,b_i,\mu_i,\E_d^{(i)})$ for  $i=1,\ldots, d$, then clearly \linebreak $\{\bZ\in \MRV^*(\alpha_i,b_i,\mu_i,\E_d^ {(i)}); i=1,\ldots, d; \Delta=d \}$. Thus, condition \eqref{cond:nullcon} provides us a little more flexibility in case we fail to have MRV on subcone $\E_d^{(j)}$ for $j>\Delta$ for some $\Delta=2,\ldots,d$.
    \item[(b)] Condition \eqref{cond:nullcon} is satisfied if $\E (Z_{(i)})^{i(\alpha_1+\widetilde \gamma)}<\infty$ {for some $\widetilde \gamma$ with $\gamma<\widetilde \gamma<\alpha_1/d$}; here $Z_{(1)}\ge \ldots \ge Z_{(d)}$ are the  order statistics of the elements of $\bZ=(Z_1,\ldots, Z_d)$. In particular, one such example is when $Z_{(j)}=0$ for $j>\Delta$, see \Cref{subsec:nohrvinallcone}. For further examples of multivariate heavy-tailed distributions exhibiting such a property see \cite{das:resnick:2015, das:fasen:2018}.
    \item[(c)] If $\Delta<d$, then \eqref{cond:nullcon} still allows for MRV to hold on $\E_d^{(i)}, i= \Delta+1,\ldots, d$ albeit with a lighter regularly varying tail rate than $i(\alpha_1+\gamma)$.
\end{enumerate}
\end{remark}

\subsection{Tail distributions, survival copulas and asymptotic tail independence} \label{subsec:copasymind}
In this section we discuss dependence structures for $d$-dimensional random vectors, which are used to model risks, claim sizes, or increments in general. Following common practice we model the marginal distributions separately from the dependence structures and hence, resort to using \emph{copulas} (cf. \citep{joe:1997,nelsen:1999}). A key feature of most of the copulas we discuss is the presence of \emph{asymptotic tail independence},  implying  that the joint exceedance of a threshold by $i$ components of the random vector occur at a rate negligible compared to joint exceedance of $(i-1)$ components for some or all $i=2,\ldots, d$.

Furthermore, we will also elaborate on  multivariate regular variation properties under these copulas. To this end, for all examples in this section, we consider random vectors $\bZ=(Z_1,\ldots,Z_d)$ with identically distributed continuous marginal components with distribution function $F_{\alpha}$ where $\overline{F}_{\alpha}\in \RV_{-\alpha}$ with $\alpha>1$, and the dependence is given by the particular (survival) copula. Moreover, we fix $b_{\alpha}(t)=\overline{F}_{\alpha}^{\leftarrow}(1/t), t>1$. Note that assuming tail equivalent marginals would lead to similar conclusions but notations become cumbersome.

 Our interest  is in tail sets, hence we will often use \emph{survival copulas} along with copulas which we recall briefly here. For a random vector $\bZ=(Z_1,\ldots, Z_d)\sim F$ with  continuous marginal distributions $F_1,\ldots, F_d$, the copula $C:[0,1]^d \to [0,1]$ and the survival copula $\widehat{C}:[0,1]^d \to [0,1]$ are {distribution functions} such that:
\begin{align*}
    F(x_1,\ldots,x_d):=\P(Z_1\le x_1,\ldots, Z_d \le x_d) & = C(F_1(x_1), \ldots, F_d(x_d)), \; (x_1,\ldots,x_d) \in \R^d,\\
   \overline{F}(x_1,\ldots,x_d):= \P(Z_1 > x_1,\ldots, Z_d > x_d) & = \widehat{C}(\overline{F}_1(x_1), \ldots, \overline{F}_d(x_d)), \; (x_1,\ldots,x_d) \in \R^d,
\end{align*}
where $\overline{F}_j=1-F_j\; \forall\, j\in\mathbb{I}.$

\begin{example}[Independence copula]\label{ex:cop:ind}
A widely used copula to exhibit {asymptotic tail independence} and {hidden regular variation} is the independence  copula. The independence copula ${C}_{\perp}$ and survival copula $\widehat{C}_{\perp}$ are given by
\begin{align}\label{eq:scop:ind}
   {C}_{\perp}(u_1,\ldots,u_d) = \widehat{C}_{\perp}(u_1,\ldots,u_d) = u_1u_2\cdots u_d, \quad 0 < u_i <1.
\end{align}
Let $\bZ\sim F$ with identical (continuous) marginal $F_{\alpha}$ as defined above and dependence given by $C_{\perp}$ (or $\widehat{C}_{\perp}$). Then
\begin{align}\label{eq:mrvhrv:ind}
    \bZ\in \MRV(i\alpha,b_{\alpha}^{1/i}, \mu_i,\E_d^{(i)})
\end{align}
where
\begin{eqnarray} \label{eq:mrvhrv:mui}
            \mu_i\left(\{\bz\in \E_d^{(i)}:z_j>x_j\; \forall\, j \in S\}\right)=\prod_{j\in S} x_j^{-\alpha}
           \end{eqnarray}
for $S\subseteq \mathbb{I}$ with $|S|=i, i =1,\ldots, d$ (cf. \cite{das:fasen:kluppelberg:2022}) and $\mu_i(\E_d^{(i+1)})=0$. Clearly $\bZ$ exhibits hidden regular variation on all cones $\E_d^{(i)}, i=2,\ldots, d$.

\end{example} 
\begin{example}[Marshall-Olkin copula]\label{ex:cop:mo}
In reliability theory, the Marshall-Olkin distribution provides an elegant mechanism to capture the dependence between the failure of subsystems in an entire system. We focus on a particular structure of the Marshall-Olkin survival copula as given in \citep[eq. (2.4), page 58]{lin:li:2014}. Assume that for all $\emptyset \neq S\subseteq \mathbb{I}$  there exists a parameter $\lambda_S>0$. Consider then the generalized Marshall-Olkin survival copula  given by
\begin{align}\label{eq:scop:mo}
   \widehat{C}_{\MO}(u_1,\ldots,u_d) = \prod_{i=1}^d \prod_{|S|=i} \bigwedge_{j\in S} u_j^{\eta_j^S}, \quad 0 < u_j <1,
\end{align}
where
\begin{align}\label{eq:MO:eta}
    \eta_j^S = \frac{\lambda_S}{\sum\limits_{J \supseteq \{j\} } \lambda_J}, \quad j\in S \subseteq \mathbb{I}.
\end{align}
A typographical error in the formula for $\eta_j^S$ in \citep[eq. (2.4), page 58]{lin:li:2014} is corrected in  \eqref{eq:MO:eta}. We consider two particular choices of the parameters $\lambda_S$ for our examples.

\begin{enumerate}[(a)]
    \item \emph{Equal parameter for all sets:} Let $\lambda_S=\lambda>0$ for all $\emptyset\neq S\subseteq \mathbb{I}$; hence, from \eqref{eq:MO:eta} we have \begin{align}\label{def:beta}
        \eta_j^S=2^{-(d-1)} =: \beta.
    \end{align} Therefore
    \begin{align*}
       \P(Z_1 > x_1,\ldots, Z_d > x_d) & = \widehat{C}(\overline{F}_{\alpha}(x_1), \ldots, \overline{F}_{\alpha}(x_d))\\
       & = \prod_{i=1}^d \prod_{|S|=i} \bigwedge_{j\in S} (\overline{F}_{\alpha}(x_j))^{\beta}\\
       & = \prod_{j=1}^d (\overline{F}_{\alpha}\left(x_{(j)})\right)^{2^{d-j}\beta} = \prod_{j=1}^d (\overline{F}_{\alpha}\left(x_{(j)})\right)^{2^{-(j-1)}},
    \end{align*}
    where $x_{(1)}\ge \ldots \ge x_{(d)}$ denote the decreasing order statistics of $x_1, \ldots, x_d$. Now, we can check that for $i=1,\ldots, d$,
    \begin{align*} 
    \bZ\in \MRV(\alpha_i, b_i, \mu_i,\E_d^{(i)})
\end{align*}
where
\begin{equation}\label{eq:MO:alphakbkmukeq}
\begin{aligned}
  \alpha_i &=(2-2^{-(i-1)})\alpha,\\
  b_i(t) &= (b_{\alpha}(t))^{\alpha/\alpha_i} = (b_{\alpha}(t))^{1/(2-2^{-(i-1)})},\\
   \mu_i & \left( \{\bz \in \E_d^{(i)}:z_j>x_j\;\forall\, j\in S\}\right) =\prod_{j=1}^i \left({{x}_{(j)}}\right)^{-\alpha2^{-(j-1)}},
\end{aligned}
\end{equation}
where $S\subseteq \mathbb{I}$, $|S|=i$ with $x_j>0$ for $j\in S$, and ${x}_{(1)}\ge \ldots \ge {x}_{(i)}$ denote the decreasing order statistics of $(x_j)_{j\in S}$
 and $\mu_i(\E_d^{(i+1)})=0$.

\item \emph{Parameters proportional to cardinality of the sets:} Let $\lambda_S=|S|\,\lambda$ where $\lambda>0$ for all $\emptyset\neq S\subseteq \mathbb{I}$. From \eqref{eq:MO:eta} we have
    $\eta_j^S=|S|(d+1)2^{-(d-1)}= |S|(d+1)\beta$ using the definition in \eqref{def:beta}. Following a similar logic as in part (a), we obtain in this case
    \begin{align*}
       \P(Z_1 > x_1,\ldots, Z_d > x_d) & =  \prod_{j=1}^d \left(\overline{F}_{\alpha}(x_{(j})\right)^{\left(1-\frac{j-1}{d+1}\right)2^{-(j-1)}}.
    \end{align*}
Again we can check that for $i=1,\ldots, d$,
      \begin{align*}
    \bZ\in \MRV({\alpha}_i^*, {b}_i^*, {\mu}^*_i,\E_d^{(i)})
\end{align*}
where,
\begin{equation}\label{eq:MO:alphakbkmukpropsize}
\begin{aligned}
  \alpha_i^* &  = \sum_{j=1}^i\Big(1-\frac{j-1}{d+1}\Big)\frac{\alpha}{2^{j-1}} = \alpha_i \frac d{d+1} +   \frac{i\alpha}{(d+1)2^{{i-1}}} ,\\
  b_i^*(t) &= (b_{\alpha}(t))^{\alpha/\alpha_i^*},\\
   \mu_i^* & \left( \{\bz \in \E_d^{(i)}:z_j>x_j \;\forall\,\ j\in S\}\right) =\prod_{j=1}^i \left({x}_{(j)}\right)^{-\alpha\left(1-\frac{j-1}{d+1}\right)2^{-(j-1)}},
\end{aligned}
\end{equation} where $|S|=i$ with $x_j>0$ for $j\in S$
 and $\mu_i^*(\E_d^{(i+1)})=0$.
\end{enumerate}

In both examples of the Marshall-Olkin copula dependence (with identical regularly varying margins), $\bZ$ exhibits hidden regular variation on all cones $\E_d^{(i)}, i=2,\ldots, d$.
\end{example}

\begin{example}[Archimedean  copula (ACIG)]\label{ex:cop:arch}
 This
 Archimedean copula example based on the Laplace transform of the inverse gamma distribution, called ACIG copula in short, appears in \cite{Hua:Joe:2011} with its hidden regular variation discussed in \cite[Example 4.4]{hua:joe:li:2014}. Suppose $\bZ=(Z_1,\ldots,Z_d)$ has an ACIG copula with dependence parameter
    $1<\beta<2$ and identical margins $\ov F_{\alpha}\in \RV_{-\alpha}$, $i=1,\ldots,d$. Then
    $\bZ\in \MRV(\alpha,b_{\alpha},\mu_1,\E_d^{(1)})$ and $\bZ\in\MRV(\alpha\beta,b_{\alpha}^{1/\beta},\mu_2,\E_d^{(i)})$  for $i=2,\ldots, d$.
 In this particular example $\bZ$ exhibits hidden regular variation on $\E_d^{(2)}$ but no further HRV at any subsequent co-ordinate subcone $\E_d^{(i)}, i \ge 3.$
\end{example}

\begin{example}[Asymptotically tail dependent copula]\label{ex:cop:dep}
In the previous examples we observed distributions with regularly varying marginals and copulas exhibiting \emph{asymptotic tail independence} leading to MRV with different indices on different spaces. But there are distributions which exhibit so-called \emph{asymptotic tail dependence} which would lead to MRV with the same index, rate function and limit measure on all subcones $\E_d^{(i)}$; see  \cite{heffernan:2000} for general examples in dimension $d=2$ and  \cite{Charpentier:Segers} for higher dimensional Archimedean copulas exhibiting asymptotic tail dependence.  We illustrate this with one example.
Let $\bZ\sim F$ such that for $\alpha>1$,
\[F(\bx) = 1- \left(1+\sum_{j=1}^d x_j^{\alpha}\right)^{-1}, \quad \bx \in \R_+^d.\]
We can check that the marginals are  {identically Pareto distributed with index $-\alpha$} and hence, the tails are $\RV_{-\alpha}$. Moreover,  $\bZ\in \MRV(\alpha, b(t)= t^{1/\alpha}, \mu, \E_d^{(i)})$ for $i=1,\ldots, d$ where 
\begin{eqnarray} \label{eq:mrvasdep:mu}
            \mu\left(\{\bz\in \E_d:z_j>x_j\; \forall\, j \in S\}\right)=\sum_{j=1}^{|S|} (-1)^{j+1} \sum_{\genfrac{}{}{0pt}{}{k_1<\ldots<k_j}{ k_1,\ldots, k_j \in S}} \left(\sum_{l=1}^j x_{k_l}^{\alpha}\right)^{-1}       
\end{eqnarray}
for $S\subseteq \mathbb{I}$.


\end{example}

\begin{remark}[Gaussian copula]
Gaussian copulas have been widely considered as a key example of asymptotic tail independence, for which coefficients of tail dependence, tail order and hidden regular variation have been studied in this context; see \cite{ledford:tawn:1996,hua:joe:li:2014,furman:kuznetsov:suzitikis:2016}. Surprisingly, the hidden regular variation properties of Gaussian copulas (with regularly varying marginal distributions) are not particularly well understood, especially in dimensions $d\ge 3$. For example, the often used Gaussian copula defined by an equi-correlation correlation matrix does not seem to admit hidden regular variation in general; see \cite{das:fasen:2022a} for details. Thus, we have refrained from using particular examples of Gaussian copulas here.
\end{remark}

\section{Aggregating regularly varying random vectors}\label{sec:mainresult}
In \Cref{sec:intro}, we discussed the principle of one large jump determining the behavior of aggregates of multivariate regularly varying random vectors in the classical framework; here we extend the idea for a more general class of tail events. We start by assuming that individual random vectors have tail equivalent margins and they admit adapted multivariate regular variation  on cones $\E_d^{(i)}, i=1,\ldots, d$ (see Definition \ref{def:amrv}). In our first result, \Cref{mainTheorem}, we consider two independent random vectors which are not necessarily identically distributed and assess the tail behavior of the sum for various tail sets. This theorem forms the basis of many subsequent results where we do assume the underlying vectors to be identically distributed as well.

\begin{theorem} \label{mainTheorem}
Let $\bZ^{(1)}, \bZ^{(2)} \in \R_+^d$ be independent random vectors, each with tail equivalent marginal distributions and $\{\bZ^{(k)}\in \MRV^*(\alpha_i^{(k)},b_i^{(k)},\mu_i^{(k)},\E_d^{(i)}), i=1,\ldots,d; \Delta_k\}$ for $k=1,2$, i.e., they are adapted-MRV on $\R^d_{+}$. Define $\alpha_0^{(k)}=0, \, b_0^{(k)\leftarrow}(t)\equiv 1,\, \mu_0^{(k)}\equiv 1$ and
\[ I(i): =   \,\mathrm{argmax}_{j \in \{0,\ldots,i\}}  \{\bar{c}_j^{(i)}: \bar{c}_j^{(i)}<\infty\}\]
where
\[ \bar{c}_j^{(i)}:= \max_{0\le m\le i}\left\{\limsup_{t\to\infty} \frac{b_{j}^{(1)\leftarrow}(t)b_{i-j}^{(2)\leftarrow}(t)}{b_{m}^{(1)\leftarrow}(t)b_{i-m}^{(2)\leftarrow}(t)} \right\}.\]
Define for $i=1,\ldots, d$, and $m=0,\ldots, i$:
$$c_m^{I(i)}:=\lim_{t\to\infty}\frac{b_{I(i)}^{(1)\leftarrow}(t)b_{i-I(i)}^{(2)\leftarrow}(t)}{b_{m}^{(1)\leftarrow}(t)b_{i-m}^{(2)\leftarrow}(t)}.$$
%
Suppose that for  $k=1,2, $ and each $m=1,\ldots,i-1$, either
$c_m^{I(i)}=0$ or  $\lim_{t\to\infty}b_m^{(k)\leftarrow}(t)/b_{m+1}^{(k)\leftarrow}(t)=0$.
 Then
 \beao
    \{\bZ^{(1)}+\bZ^{(2)}\in\MRV^*(\alpha_{i},b_{i},\mu_{i}^{\oplus},\E_d^{(i)}), i=1,\ldots,d;\Delta^{\oplus}\}
 \eeao
 with $\alpha_i=  \alpha_{I(i)}^{(1)} +\alpha_{i-I(i)}^{(2)}, b_{i}^{\leftarrow}(t)=b_{I(i)}^{(1)\leftarrow}(t)b_{i-{I(i)}}^{(2)\leftarrow}(t), $ and
 \beao
    \mu_i^{\oplus}(A)=\sum_{m=0}^i c_m^{I(i)}\mu_{m,i}^*(A) \quad \text{ for }A\in\mathcal{B}(\E_d^{(i)}),
 \eeao
 where $\mu_{m,i}^*$ is the measure which is uniquely defined on $\cR^{(i)}$ as follows: for \linebreak $A=\{\bz \in \R_+^d: z_j>x_j \;\forall\; j\in S\}\in \cR^{(i)}$ with $S\subseteq \mathbb{I}$, $|S|\ge i$ and $x_j>0$   $\forall j\in S$ we have 
 \begin{align*}
        &\mu_{m,i}^*\left(A\right)\\
         &=\sum_{\genfrac{}{}{0pt}{}{J\subseteq S\cup \{\emptyset\}}{|J|=m}} \mu_{m}^{(1)}\Bigg(\{\bz\in\E_d^{(m)}: z_{j}>x_{j}\;\forall\;j\in J\}\Bigg)
        \mu_{i-m}^{(2)}\,\Bigg(\{\bz\in\E_d^{(i-m)}:  z_{j}>x_{j}\;\forall\; j\in S\backslash J\}\Bigg).
 \end{align*}
 Moreover, $\Delta^{\oplus} \in \{\max(\Delta_1+1,\Delta_2+1), \ldots, \min(\Delta_1+\Delta_2,d)\}$.
 \end{theorem}

The proof of \Cref{mainTheorem} is given in \Cref{sec:proofmaintheorem}.

\begin{remark}
Since the output of $\text{argmax}$ may contain multiple elements, $I(i)$ is not defined uniquely; hence, a value for $I(i)$ is often chosen from these outputs  according to convenience.
\end{remark}

We have refrained from stating a general result akin to \Cref{mainTheorem} for adding $n$ random vectors since the parameters of the limit model become notationally cumbersome without providing additional insight; on the other hand, for a variety of joint dependence behavior, we often observe nicer structures appearing.
In the rest of the section, we discuss consequences of \Cref{mainTheorem} on the finite sum of  i.i.d.  random vectors under various assumptions on their dependence structures.

\subsection{All subcones exhibit regular variation}\label{subsec:hrvinallcones}
First, we investigate the case where we add  i.i.d.  random vectors which are multivariate regularly varying on all relevant cones. The results as we will see are direct consequences of \Cref{mainTheorem}. We begin with the well-known model where all components of each vector are  i.i.d.  random variables as well; \Cref{ex:cop:ind} gives the structure of the limit measure in this case. The following proposition provides a slightly general version of this case. Proofs of the results of this subsection are available in \Cref{app:subsec:hrvinallcones}.

\begin{proposition}[Nearly independent case]\label{prop:ind}
Let $\bZ^{(1)},\ldots,\bZ^{(n)}$ be  i.i.d.  random vectors in $\R_+^d$ with tail equivalent marginal distributions and $\bZ^{(1)}\in\MRV(\alpha_i,b_i,\mu_i,\E_d^{(i)})$  for $i=1,\ldots,d$ where $b_i(t)=(b_1(t))^{1/i}$ and $b_1(t)\in \RV_{1/\alpha}$. Then $\alpha_i=i\alpha$ and
         \begin{align}\label{mrv:ind:sum}
               & \sum_{k=1}^n\bZ^{(k)}\in\MRV(\alpha_i,b_i,\E_d^{(i)}) \quad  \text {for } \quad i=1,\ldots,d.
         \end{align}
    Now if for some $\kappa_j>0$, $j=1,\ldots,d$,
           \begin{eqnarray} \label{star}
            \mu_i\left(\{\bz\in \E_d^{(i)}:z_j>x_j\; \forall\, j \in S\}\right)=\prod_{j\in S}\kappa_j x_j^{-\alpha}
           \end{eqnarray}
         for $S\subseteq \mathbb{I}$ with $|S|=i$,   $x_j>0$ and $\mu_i(\mathbb{E}_d^{(i+1)})=0$, $i=1,\ldots, d$, {then }
         \beao                \sum_{k=1}^n\bZ^{(k)}\in\MRV(i\alpha,b_1^{1/i},n^{i}\mu_i,\E_d^{(i)}) \quad  \text {for } \quad i=1,\ldots,d.
         \eeao
\end{proposition}

\begin{remark}
If all components of the random vectors $\bZ^{(1)},\ldots,\bZ^{(n)}$ are completely tail equivalent then $\kappa_1=\kappa_2=\ldots=\kappa_d$.
\end{remark}
{Although  condition \eqref{star} is rather restrictive, the result obtained in \eqref{mrv:ind:sum}, i.e., \linebreak if $\bZ^{(1)}\in\MRV(\alpha_i,b_i ,\E_d^{(i)})$ then $\sum_{k=1}^n\bZ^{(k)}\in\MRV(\alpha_i,b_i,\E_d^{(i)})$, holds under much weaker assumptions. The  particular assumption \eqref{star} 
helps only to calculate the exact form of the limit distribution. The following result provides a further case and helps in creating many examples.}

\begin{proposition}\label{prop:mrv:tailsremainsame}
Let $\bZ^{(1)},\ldots,\bZ^{(n)}$ be  i.i.d.  random vectors in $\R_+^d$ with tail equivalent marginal distributions and  $\bZ^{(1)}\in\MRV(\alpha_i,b_i,\mu_i,\E_d^{(i)})$  for $i=1,\ldots,d$.
Moreover, assume that $\alpha_i<\alpha_m+\alpha_{i-m}$ for all $m= 1,\ldots, i-1$ and $i = 2,\ldots,d$. Then
 \beao
                \sum_{k=1}^n\bZ^{(k)}\in\MRV(\alpha_i,b_i,n\mu_i,\E_d^{(i)}) \quad  \text {for } \quad i=1,\ldots,d.
         \eeao
\end{proposition}

\begin{remark}
Clearly,  a sufficient condition for \Cref{prop:mrv:tailsremainsame} to hold would be to assume that $c_m^{I(i)}=0, m=1, \ldots, i-1$ and $i = 2,\ldots,d$ instead of  $\alpha_i<\alpha_m+\alpha_{i-m}$ for all $m= 1,\ldots, i-1$ and $i = 2,\ldots,d$. This requires the notation of \Cref{mainTheorem}, and we prefer the latter in lieu of interpretability.
\end{remark}

\begin{remark}
Both in  Propositions \ref{prop:ind} and  \ref{prop:mrv:tailsremainsame}, we observe that while adding finitely many  random vectors $\bZ^{(1)}, \ldots, \bZ^{(n)}$, we obtain $\sum_{k=1}^n \bZ^{(k)} \in \MRV(\alpha_i,b_i, \mu_{i,n}^{\oplus}, \E_d^{(i)})$. The  indices of regular variation $\alpha_i$ and the scaling parameter $b_i$ remain the same no matter how many vectors we add although the measure $\mu_{i,n}^{\oplus}$ are quite different for different values of $n$. Note the following.
\begin{enumerate}[i)]
    \item Under the assumptions of \Cref{prop:ind}, we have, $\alpha_i=\alpha_{m}+\alpha_{i-m}$ for {$m=0,\ldots,i$,}  and  $\mu_{i,n}^{\oplus}=n^i\mu_i$. Interestingly,  $\alpha_i=\alpha_{m}+\alpha_{i-m}$ for {$m=0,\ldots, i$} does not necessarily imply that $\mu_{i,n}^{\oplus}=n^i\mu_i$.
    \item Under the assumptions of \Cref{prop:mrv:tailsremainsame}, we have, $\alpha_i <\alpha_{m}+\alpha_{i-m}$ for {$m=1,\ldots,i-1$,}  which turns out to be a sufficient condition for $\mu_{i,n}^{\oplus}=n\mu_i$.
\end{enumerate}
\end{remark}

\begin{example}[Marshall-Olkin dependence] \label{example:Mashall-Olkin}
For this example let $\bZ,\bZ^{(1)},\ldots,\bZ^{(n)}$ be  i.i.d.  random vectors in $\R_+^d$ with all marginal distributions following $F_{\alpha}$ where $\overline{F}_{\alpha}\in \RV_{-\alpha}$ and $b_{\alpha}(t)=\overline{F}^{\leftarrow}(1/t)$ for some $\alpha>0$. We consider two different Marshall-Olkin dependence parameters for $\bZ$ which has copula $\widehat{C}_{\MO}$ given by \eqref{eq:scop:mo}; see \Cref{ex:cop:mo}.
\begin{enumerate}[(a)]
    \item \emph{Equal parameter:} Suppose the parameters  defining the Marshall-Olkin copula are given by \eqref{def:beta}. Then from \Cref{ex:cop:mo}(a) we have $ \bZ\in \MRV(\alpha_i, b_i, \mu_i,\E_d^{(i)})$ for $i=1,\ldots, d$ where  {$\alpha_i, b_i, \mu_i$ } are given by \eqref{eq:MO:alphakbkmukeq}. Now observe that for fixed  $i=1,\ldots, d$, and \linebreak $m=1,\ldots, i-1$,
    \begin{align}
        \alpha_{m}+\alpha_{i-m} & = (2-2^{-(m-1)})\alpha + (2-2^{-(i-m-1)})\alpha \nonumber\\
        & \ge 2\alpha \nonumber\\
        &> (2-2^{-(i-1)})\alpha = \alpha_i. \label{eq:alphaorder}
    \end{align}
    Hence, by \Cref{prop:mrv:tailsremainsame}, we have
     \beao
                        \sum_{k=1}^n\bZ^{(k)}\in\MRV(\alpha_i,b_i,n\mu_i,\E_d^{(i)}).
                \eeao
 \item \emph{Proportional parameter:}  We know from \Cref{ex:cop:mo}(b) that $ \bZ\in \MRV(\alpha_i^*, b_i^*, \mu_i^*,\E_d^{(i)})$ for $i=1,\ldots, d$ where $\alpha_i^*, b_i^*, \mu_i^*$ are given by \eqref{eq:MO:alphakbkmukpropsize}. Again note that for fixed $i=1,\ldots, d$, $m=1,\ldots, i-1$, and $\alpha_i$  as in \eqref{eq:MO:alphakbkmukeq},
    \begin{align*}
        \alpha_{m}^*+\alpha_{i-m}^* & = \alpha_m\frac{d}{d+1}+ \frac{m\alpha}{(d+1)2^{m-1}}+ \alpha_{i-m}\frac{d}{d+1}+ \frac{(i-m)\alpha}{(d+1)2^{i-m-1}}\\
        & = (\alpha_m+\alpha_{i-m})\frac{d}{d+1} + \frac{\alpha}{(d+1)}\left(\frac{m}{2^{m-1}} + \frac{i-m}{2^{i-m-1}}\right)\\
        & > \alpha_i\frac{d}{d+1} + \frac{\alpha}{(d+1)}\left(\frac{m}{2^{i-1}} + \frac{i-m}{2^{i-1}}\right) \quad \text{(using \eqref{eq:alphaorder})}\\
       & = \alpha_i^*.
    \end{align*}
    Hence, by \Cref{prop:mrv:tailsremainsame}, we have
     \beao
                        \sum_{k=1}^n\bZ^{(k)}\in\MRV(\alpha_i^*,b_i^*,n\mu_i^*,\E_d^{(i)}).
                \eeao
\end{enumerate}
\end{example}

\begin{example}[Archimedean copula] \label{example:Archimedean}
Referring to \Cref{ex:cop:arch}, suppose we have  i.i.d.  random vectors $\bZ^{(1)}, \ldots, \bZ^{(n)}$ with identical margins $F_{\alpha}$ so that $\overline{F}_{\alpha}\in \RV_{-\alpha}, b_{\alpha} = \overline{F}^{\leftarrow}(1/t)$ and they admit an ACIG copula as described with dependence parameter $1<\beta<2$. Then we have $\bZ^{(1)}\in \MRV(\alpha,b_{\alpha},\mu_1,\E_d^{(1)})$ and {$\bZ^{(1)}\in\MRV(\alpha\beta,b_{\alpha}^{1/\beta},\mu_2,\E_d^{(i)})$ } for $i=2,\ldots, d$. Now, clearly the conditions for \Cref{prop:mrv:tailsremainsame} are satisfied and hence, we have $\sum_{k=1}^n\bZ^{(k)}\in \MRV(\alpha,b_{\alpha},n\mu_1,\E_d^{(1)})$ and {$\sum_{k=1}^n\bZ^{(k)} \in\MRV(\alpha\beta,b_{\alpha}^{1/\beta},n\mu_2,\E_d^{(i)})$ } for $i=2,\ldots, d$.
\end{example}

The two extreme cases of dependence considered in general are the case of fully independent components for $\bZ^{(1)}$, which is covered in \Cref{prop:ind}, and the case where the components of $\bZ^{(1)}$ are dependent such that $\bZ^{(1)}\in\MRV(\alpha,b,\mu,\E_d^{(i)})$ for $i=1,\ldots,d$. The following corollary addresses the latter case. 

\begin{corollary}[Dependent case, corollary to Proposition \ref{prop:mrv:tailsremainsame}]\label{cor:asymdep}
Let $\bZ^{(1)},\ldots,\bZ^{(n)}$ be  i.i.d.  random vectors in $\R_+^d$ {with tail equivalent marginal distributions }
and \linebreak $\bZ^{(1)}\in\MRV(\alpha_i,b_i,\mu_i,\E_d^{(i)})$  for $i=1,\ldots,d$. Moreover,  $(\alpha_i,b_i,\mu_i,\E_d^{(i)}) = (\alpha,b,\mu,\E_d^{(i)})$ for $i=1,\ldots, i^*$ for some $ i^*\le d$. Then
\beao
                \sum_{k=1}^n\bZ^{(k)}\in\MRV(\alpha,b,n\mu,\E_d^{(i)}) \quad  \text {for } \quad i=1,\ldots,  i^*.
         \eeao
\end{corollary}

\begin{example}
 \Cref{ex:cop:dep} exhibiting asymptotic tail dependence admits the property $\bZ^{(1)}\in\MRV(\alpha,b,\mu,\E_d^{(i)})$ for $i=1,\ldots, d$  and hence, we can compute tail asymptotics  of its convolution using \Cref{cor:asymdep}.
\end{example}
In all the propositions, corollaries and hence examples of this section, we observe that if $\bZ^{(1)} \in \MRV(\alpha_i, b_i, \mu_i, \E_d^{(i)})$, then their finite sum $\sum_{k=1}^n \bZ^{(k)} \in \MRV(\alpha_i, b_i, C_{n,i} \mu_i, \E_d^{(i)})$
for some constant $C_{n,i}>0$. The ``one large jump'' phenomenon is observed here, in the sense of equivalence of the first and last expressions in \eqref{eq:nMRV} or \eqref{eq:nMRVgeneral} if $C_{n,i}=n$; cf. \Cref{prop:mrv:tailsremainsame}. But $C_{n,i}$ is not necessarily $n$, as for example in \Cref{prop:ind}, and this is a case which we think of as a phenomenon of ``more than one large jump'' or ``a few large jumps''.  Note that in both cases mentioned here, we did assume  $\bZ^{(1)}$ to have MRV on all cones, but not be strictly adapted-MRV. In \Cref{subsec:nohrvinallcone}, we illustrate that a similar principle holds, even under the assumption of adapted-MRV, although the characterizing jumps are now of the form \eqref{eq:nMRVgeneral2} which relates to ``a few large jumps'' phenomenon.

\subsection{Not all subcones necessarily exhibit regular variation} \label{subsec:nohrvinallcone}
In certain contexts, we may be interested in adding random vectors which are not necessarily MRV on all relevant cones. For example, we may have a sequence of  i.i.d.  random vectors for which not all components are non-zero in each realization. An extension of such aggregation to random sums lead to general compound Poisson or L\'evy processes, see \Cref{sec:randomsum} for details. In this section, we concentrate on a few such examples. The general structure for the limit measures in such problems are often not quite apparent.

\begin{proposition}\label{prop:Delta1}
Let $\bZ^{(1)},\ldots, \bZ^{(n)} \in \R_+^d$ be  i.i.d.   random vectors
with  tail equivalent marginal distributions which are $\RV_{-\alpha}$  and let $\{\bZ^{(1)}\in \MRV^*(\alpha_i,b_i,\mu_i,\E_d^{(i)}), i=1,\ldots,d; \Delta=1\}$
 with $\alpha_1=\alpha$. Then $$
 \left\{\sum_{k=1}^n\bZ^{(k)}\in \MRV^*(\alpha_{i,n},b_{i,n},\mu_{i,n}^{\oplus},\E_d^{(i)}), i=1,\ldots,d;\Delta= \min\{d,n\}\right\}.$$ Specifically, for $i=1,\ldots,d$ we have the following:
\begin{enumerate}[(a)]
    \item For $n\geq i$,
    \begin{align}\label{eq:mrv:muiplusdelta1}
     \sum_{k=1}^n\bZ^{(k)}\in\MRV\left(\alpha_{i,n}=i\alpha,b_{i,n}=b_1^{1/i},\mu_{i,n}^{\oplus}, \, \E_d^{(i)}\right),
    \end{align}
    where
    \begin{align}\label{eq:limit:muiplusdelta1}
    \mu_{i,n}^{\oplus}\Big(\{\bz\in \E_d^{(i)}:z_j>x_j\; \forall\, j \in S\}\Big)= \frac{n!}{(n-i)!}\prod_{j\in S}
    \mu_1\left(\{\bz\in \E_d^{(i)}:z_j>x_j\}\right),
    \end{align}
    for $S\subseteq \mathbb{I}$ with $|S|=i$, $x_j>0$ for $j\in S$ {and $\mu_{i,n}^{\oplus}(\mathbb{E}_d^{(i+1)})=0$}.
    \item For $1\leq n<i$,
     \begin{align}\label{eq:mrv:nciplusdelta1}
     \sum_{k=1}^n\bZ^{(k)}\in\mathcal{NC}\left(b_{i,n}(t)=t^{1/i(\alpha+\gamma)}, \E_d^{(i)}\right).
    \end{align}
\end{enumerate}
\end{proposition}
The proof of \Cref{prop:Delta1} is in \Cref{sec:proofpropDelta1}.
\begin{remark}\label{rem:delta1} In \Cref{prop:Delta1}, if the marginal distributions are completely tail equivalent with distribution functions $F_j, j=1,\ldots, d$ and $b_1(t)=\overline{F}_1^{\leftarrow }(1/t)$, then $\mu_{i,n}^{\oplus}$ in \eqref{eq:limit:muiplusdelta1} is given by
 \begin{eqnarray*}
           \mu_{i,n}^{\oplus}\Big(\{\bz\in \E_d^{(i)}:z_j>x_j\; \forall\, j \in S\}\Big)=\frac{n!}{(n-i)!}\prod_{j\in S} x_j^{-\alpha}.
           \end{eqnarray*}
\end{remark}

\begin{remark}\label{rem2:delta1}
For the conclusion of \Cref{prop:Delta1} to hold,
the random variables $\bZ^{(1)}, \bZ^{(2)},\ldots$ need not be identically distributed as long as they are independent and are all adapted-MRV with the same sets of parameters. The proof follows by similar arguments as the proof of \Cref{prop:Delta1} and is skipped.
\end{remark}

 The phenomenon of a ``few large jumps" holds here too, as illustrated next. Assume that in Proposition \ref{prop:Delta1}, the marginal distributions are completely tail equivalent as in \Cref{rem:delta1}; and $\bZ=(Z_{1},\ldots,Z_{d})\sim \bZ^{(1)}$. Without loss of generality let $tA$ be the tail event of interest where $A=\{\bz\in \E_d^{(i)}:z_j>x_j\; \forall\; j\in\{1,\ldots, i\}\}$. Note that from \eqref{eq:mrv:muiplusdelta1} and \eqref{eq:limit:muiplusdelta1}, we can infer that in fact as $t\to \infty$,
\begin{align*}
\P(\bZ^{(1)}+\ldots+\bZ^{(n)} \in tA) & \sim \frac{1}{(b_1^{\leftarrow}(t)^{i})}
\frac{n!}{(n-i)!} \prod_{j=1}^i x_j^{-\alpha}\\
        & = C_{n,i}  \prod_{j=1}^i \frac{x_j^{-\alpha}}{b_1^{\leftarrow}(t)} \\
        &\sim C_{n,i}  \prod_{j=1}^i \P(Z_{j}>tx_i),
\end{align*}
where $C_{n,i}= (n!)/((n-i)!)$. Hence $\bZ^{(1)}+\ldots+\bZ^{(n)} \in tA$ occurs at the same rate with which  $i$ independent univariate marginals cross their respective thresholds, indicating $i$ many large jumps. The constant $C_{n,i}$ gives the number of possible choices of independent jumps, here the marginals jumps counted are all from different variables $\bZ^{(k)}$ .

In the rest of the section we provide examples exhibiting \Cref{prop:Delta1} and its possible generalisation.

\begin{example}\label{ex:mrvsum:delta1}
Let $(X^{(k)})_{k\in\N}$ be a sequence of i.i.d.  random variables with distribution function $F$ and $\overline{F}\in \RV_{-\alpha}, \alpha>0$. Let $(\bB^{(k)})_{k\in\N}$ be  i.i.d.  random vectors  taking values in $\{e_1,\ldots,e_d\}$ with $\P(\bB^{(1)}=e_l)=p_l\ge 0, \sum_l p_l=1$, and $e_l=(0,\ldots,0,1,0,\ldots,0)  \in \{0,1\}^d$ where the only non-zero entry $1$ is at the $l$-th place. Define $\bY^{(k)}:= X^{(k)}\bB^{(k)}, k\in\N$. Moreover, let $(\bepsilon^{(k)})_{k\in\N}$ be an i.i.d. sequence of random vectors with $\E\|\bepsilon^{(k)}\|^{d(\alpha+\theta)}<\infty$ for some $\theta>0$. Finally, also assume that
 $X^{(1)},X^{(2)},\ldots ,\bB^{(1)},\bB^{(2)},\ldots,\bepsilon^{(1)},\bepsilon^{(2)},\ldots$ are independent.
Then $\bZ^{(k)}:= \bY^{(k)}+\bepsilon^{(k)}$, $k\in\N$, are a sequence of i.i.d. adapted-MRV random vectors with $\Delta=1$ (cf. \Cref{Lemma_aux}) and hence, \Cref{prop:Delta1} (along with \Cref{rem2:delta1}) provides the tail asymptotic behavior of $\sum_{k=1}^n \bZ^{(k)}$ for any $n\ge 1$.
\end{example}

The neat expressions for limit measures and tail indices as obtained by \Cref{prop:Delta1} in aggregating i.i.d adapted-MRV random vectors with $\Delta=1$ does not extend as nicely for $\Delta>1$. Nevertheless, we may still be able to find a pattern in certain cases and our next example with $\Delta=2$ elaborates on this.

\begin{example} \label{ex:mrvsum:delta2}
The setting is similar to \Cref{ex:mrvsum:delta1}. Let $X^{(1)},X^{(2)},\ldots, $ \linebreak $\widetilde X^{(1)},\widetilde X^{(2)},\ldots \sim F$ be  i.i.d.  random variables   with $\overline{F}\in \RV_{-\alpha}, \alpha>0$. Let $
\bB^{(1)}, \bB^{(2)},\ldots, $ \linebreak
$\widetilde \bB^{(1)},\widetilde \bB^{(2)},\ldots$ be  i.i.d.  random vectors  taking values in $\{e_1,\ldots,e_d\}$ as defined in \Cref{ex:mrvsum:delta1}.
Also assume that
 $X^{(1)},X^{(2)},\ldots,\widetilde X^{(1)},\widetilde X^{(2)},\ldots,\bB^{(1)},\bB^{(2)},\ldots,\widetilde \bB^{(1)},\widetilde \bB^{(2)},\ldots$ are mutually independent. Then $\bY^{(k)}:= 2^{-1/\alpha}X^{(k)}\bB^{(k)}+2^{-1/{\alpha}}\widetilde X^{(k)}\widetilde \bB^{(k)}, k\in\N$, are  i.i.d.  adapted-MRV random vectors with $\Delta=2$. Specifically for any $k\ge 1$:
 \begin{enumerate}
\item[i)] $\bY^{(k)}\in \MRV(\alpha, b_1, \mu_1, \E_d^{(1)})$ where $b_1(t)=\overline{F}^{\leftarrow}(1/t)$ and with $A_1=\{\bz\in \E_d^{(i)}: z_j>x_j\}\in\mathcal{R}^{(1)}$ for some $j\in \mathbb{I}$, $\mu_1(A_1)= p_jx_j^{-\alpha}$.
\item[ii)] $\bY^{(k)}\in \MRV(2\alpha, b_1^{1/2}, \mu_2, \E_d^{(2)})$ where with $A_2=\{\bz\in \E_d^{(2)}: z_j>x_j,\,z_{\ell}>x_{\ell} \}\in \mathcal{R}^{(2)}$ for some $j,\ell\in \mathbb{I}$, $j\not=\ell$ we have  $\mu_2(A_2)= \frac 12 p_jp_{\ell}(x_jx_{\ell})^{-\alpha}$. 
\item[iii)] For $i=3,\ldots, d$ and some $0<\gamma<\alpha/d$, we have $\bY^{(k)}\in \mathcal{NC}(t^{1/i(\alpha+\gamma)}, \E_d^{(i)})$.
\end{enumerate}
Applying \Cref{mainTheorem}, and following the proof of \Cref{prop:Delta1}, we can show that for $n\ge i$,
\[\sum_{k=1}^n \bY^{(k)} \in \MRV(i\alpha, b_1^{1/i}, \mu_{i,n}^{*\oplus}, \E_d^{(i)} ),\]
where
\begin{align}
    \mu_{i,n}^{*\oplus}\Big(\{\bz\in \E_d^{(i)}:z_j>x_j\; \forall\, j \in S\}\Big) & = f_i(n)\prod_{j\in S}
    \mu_1\left(\{\bz\in \E_d^{(i)}:z_j>x_j\}\right) \nonumber \\ & = f_i(n)\prod_{j\in S} {p_j} x_j^{-\alpha}, \label{eq:limit:muiplusdelta2}
    \end{align}
    for $S\subseteq \mathbb{I}$ with $|S|=i$, $x_j>0$ for $j\in S$ and some function $f_i:\mathbb{N}\to \R_+$ where $$ \frac{n!}{(n-i)!} \le f_i(n) \le n^i, \quad n\ge i, \; i\in \mathbb{I}.$$
Furthermore, $\mu_{i,n}^{*\oplus}(\E_d^{(i+1)})=0$. In particular, we can check that
 \begin{align*}
    f_1(n) &= n, \; n\ge 1, \quad& f_2(n)&=  n(n- 1/2),\; {n\ge 1},\\
    f_3(n) &= n(n-1/2)(n-1), \; {n\ge 2},\quad & f_4(n)&=  n(n-1/2)(n-1)(n-3/2),\; {n\ge 2}.
\end{align*}
A pattern in the  value of $f$ emerges for this example, but it depends on the limit measures of the underlying variables $\bY^{(k)}$. Examples in the same spirit can be computed for $\Delta\ge 3$ involving some careful  combinatorial accounting.
\end{example}

\begin{remark}
It is easy to extend \Cref{ex:mrvsum:delta2} in the spirit of \Cref{ex:mrvsum:delta1}. Suppose $\bY^{(k)},\,  k\in\N$, are the same random vectors as in \Cref{ex:mrvsum:delta2} and $(\bepsilon^{(k)})_{k\in\N}$ are i.i.d. random vectors with $\E\|\bepsilon^{(k)}\|^{d(\alpha+\theta)}<\infty$ for some $\theta>0$, which are also independent of the sequence $(\bY^{(k)})_{k\in\N}$. Then $\bZ^{(k)}:=\bY^{(k)}+\bepsilon^{(k)}$, $k\in\N$, is an adapted-MRV sequence of random vectors with $\Delta=2$.  All the conclusions for $(\bY^{(k)})_{k\in\N}$, and $\sum_{k=1}^n\bY^{(k)}, n\ge i$, in \Cref{ex:mrvsum:delta2} also hold for
$(\bZ^{(k)})_{k\in\N}$, and $\sum_{k=1}^n\bZ^{(k)}, n\ge i$, by an
application of \Cref{Lemma_aux}.
\end{remark}

\section{Random sums of regularly varying random vectors}\label{sec:randomsum}
A natural extension from aggregating finitely many random vectors is to aggregate randomly many random vectors, which we discuss in this section, finally leading towards an extension to L\'evy processes in \Cref{section:Levy}. We observed in \Cref{sec:mainresult} that the behavior of the finite sum may take various forms even when they are multivariate regularly varying. Hence, for convenience, for the rest of the paper, we assume that the following is satisfied.

\begin{assumptionletter} \label{AssumptionA}
Let $\bZ, (\bZ^{(k)})_{k\in\N}$ be a sequence of  i.i.d.  random vectors in $\R_+^d$. Assume that for all $i=1,\ldots,d$ there exists a measurable function $f_i:\N\to\R_+$ and {a non-null measure $\mu_i\in \mathbb{M}(\E_d^{(i)})$ 
  such that for any $n\in\N$, 
\beao
   \left\{\sum_{k=1}^n\bZ^{(k)}\in\MRV^*(\alpha_{i,n},b_{i,n},\mu_{i,n}^{\oplus}=f_i(n)\mu_i,\E_d^{(i)}); i=1,\ldots,d;\Delta_n\right\}
\eeao
and $\Delta_n=d$ for $n\ge d$. Furthermore, assume that for all $i=1, \ldots, d$, there exist a finite constant $\alpha_i>0$ and a regularly varying function $b_i(t)\in\RV_{1/\alpha_i}$ such that if $f_i(n)\not=0$ we have $\alpha_{i,n}=\alpha_i$ and $b_{i,n}=b_i$.} 
\end{assumptionletter}

\begin{remark} \label{Remark 1} $\mbox{}$
\begin{enumerate}
  \item[(a)]  In general, the structure of the function $f_i$ can be quite complex and often requires an involved combinatorial  accounting procedure, see  \Cref{ex:mrvsum:delta2}; nevertheless  in
    several examples we do observe that $f_i(n)=n$ and in all our examples $0\leq f_i(n)\leq n^i$. Assumption~\ref{AssumptionA} allows us the flexibility to not get involved in the computation of $f_i$.
   \item[(b)]  For $\Delta_n<d$ we have $f_i(n)=0$ and $\alpha_{i,n}=\infty$ for $i=\Delta_n+1,\ldots,d$,  and hence, we have null convergence. 
 On the other hand, for $i=1,\ldots,\Delta_n$ we have $0<f_i(n)<\infty$, $\alpha_{i,n}=\alpha_i$ and $b_{i,n}=b_i$.  For the examples considered in \Cref{subsec:nohrvinallcone} this happens to be the case.

    \item[(c)] Suppose $\Delta_1=d$, then Assumption~\ref{AssumptionA} implies that  $I(i)$ as defined
        in \Cref{mainTheorem} can be chosen to be $i$  and hence, $\alpha_m+\alpha_{i-m}\geq \alpha_i$ for every $m=0,\ldots,i$ and $i=1,\ldots,d$.
On the other hand,  $\alpha_m+\alpha_{i-m}> \alpha_i$ for every $m=1,\ldots,i-1$  is a sufficient condition for  $I(i)=i$.
\end{enumerate}
\end{remark}

 
\newpage
    
    \begin{remark} \label{Remark 2} 
    Under Assumption \ref{AssumptionA}, define $\bZ^{(\oplus,k)}:=\sum_{l=(k-1)d+1}^{kd}\bZ^{(l)}$, $k\in\N$. Also let $\bZ^{\oplus}:=(Z_{1}^{\oplus},\ldots, Z_{d}^{\oplus}) \sim \bZ^{(\oplus,1)}$ and $Z_{(1)}^{\oplus}\geq\ldots Z_{(d)}^{\oplus}$ be the order statistics of $Z_{1}^{\oplus},\ldots, Z_{d}^{\oplus}$. 
\begin{enumerate}
    \item[(a)]  From Assumption~\ref{AssumptionA} we have 
    \beao
        \sum_{k=1}^n\bZ^{(\oplus,k)}\in \MRV(\alpha_{i},b_{i},f_i(dn)\mu_i,\E_d^{(i)})
    \eeao
    with $0<f_i(dn)<\infty$ for $i=1,\ldots, d$ and $n\in\N$. Now, a consequence of \Cref{Remark 1} (c) is  that $I(i)$ (as defined
        in \Cref{mainTheorem}) for the random vector $\bZ^{(\oplus,1)}$ (or equivalently $\bZ^{\oplus}$) is equal to $i$ and hence, $\alpha_m+\alpha_{i-m}\geq \alpha_i$ for every $m=0,\ldots,i$ and $i=1,\ldots,d$. 
    Now, $I(i)=i$ implies as well that there exists a finite constant $C^*>0$ such that
    \beam \label{E6}
            0\leq \sup_{t>0}\sum_{i=1}^d\sum_{m=0}^{i}\frac{\P\left(Z^{\oplus}_{(m)}>t\right)
                \P\left(Z^{\oplus}_{(i-m)}>t\right)}{\P\left(Z^{\oplus}_{(i)}>t\right)}
                \leq C^{*}.
        \eeam
    \item[(b)] 
    Note that the function
    $g_i(t):=\P\left(Z_{(i)}^{\oplus}>t\right)\in \mathcal{RV}_{-\alpha_{i}}$ for any $i\in\mathbb{I}$. 
    Hence, using Potter's bound \citep[Proposition B.1.19 (5)]{dehaan:ferreira:2006} there exists a finite constant $C^{**}>0$ such that 
        \beam \label{E7}
             \sup_{ i\in \mathbb{I} }\sup_{t>0}
            \frac{\P\left(Z^{\oplus}_{(i)}>t/2\right)}{\P\left(Z^{\oplus}_{(i)}>t\right)}\leq C^{**}.
        \eeam  
\end{enumerate}
\end{remark}

\begin{theorem} \label{thm:randomsum:main}
Let Assumption~\ref{AssumptionA} hold and {let the i.i.d. sequence}
 $(\bZ^{(k)})_{k\in\N}$ be independent of the $\N_0$-valued random variable $\tau$ with $\E(\kappa^\tau)<\infty$ for any $\kappa>0$.
Then for $i=1,\ldots,d$,
\beao
   \sum_{k=1}^{\tau}\bZ^{(k)}\in \MRV(\alpha_i,b_i,\E(f_i(\tau))\mu_i,\E_d^{(i)}).
\eeao
\end{theorem}
The proof is in \Cref{sec:proofmainrandomsum}.
Note that the examples considered in both \Cref{subsec:hrvinallcones} and \Cref{subsec:nohrvinallcone}, all satisfy Assumption~\ref{AssumptionA} (as well as \eqref{E6}).  Hence, for any $\N_0$-valued random variable $\tau$ whose  moment generating function exists on the positive real line, we can compute the tail probability of a random sum of $\tau$ many i.i.d. MRV random vectors using \Cref{thm:randomsum:main}.

\section{Regular variation in multivariate L\'evy processes}\label{section:Levy}
A particular example of a random sum of i.i.d. random vectors as indicated in \Cref{thm:randomsum:main} is the compound Poisson process {at a fixed time point} where the number of summands $\tau$
is Poisson distributed, which in turn is an example of a L\'evy process.  In this section, we investigate multivariate regular variation
of  L\'evy processes $\bL=(\bL(s))_{s\geq 0}$ on different subcones $\E_d^{(i)}$, $i=1,\ldots,d$, and relate it to multivariate regular variation of the L\'evy measure $\Pi$  on those subcones.
 A  L{\'e}vy process   is
characterized by its {\em L{\'e}vy-Khinchine representation}
$\E(\e^{i\langle\Theta,\bL(s)\rangle})=\exp(-s\,\Psi(\Theta))$
for $\Theta\in\mathbb{R}^d$, where
 \beao
        \Psi(\Theta)=
        -i\langle \gamma,\Theta\rangle
        +\frac{1}{2}\langle\Theta,\Sigma\,\Theta \rangle+\int_{\R^d}
 \left(1-\e^{i\langle\Theta,\bx\rangle}+i\langle \bx,\Theta\rangle\right)\,\Pi(d\bx)
    \eeao
with  $\gamma\in\mathbb{R}^d$, $\Sigma$ a non-negative definite
matrix in $\mathbb{R}^{d\times d}$ and a Borel measure
 $\Pi$   on $\mathbb{R}^d$, called the L\'evy measure
 which satisfies
$\int_{\R^d}\min\{\|\bx\|^2,1\}\,\Pi(d\,\bx)<\infty$
and  $\Pi(\bzero)=0$. Moreover, $\langle
\cdot,\cdot\rangle$ denotes the inner product in $\R^d$.
The L\'evy measure $\Pi(A)$ measures the expected number of jumps of the L\'evy process in the interval $[0,1]$ which lies in the set $A$. We denote by $\Pi_j$ for $j=1,\ldots,d$ the marginal L\'evy measures. 
In this paper, we restrict to L\'evy processes in $\R_+^d$, i.e., the marginal L\'evy processes are \emph{subordinators}, which are increasing L\'evy processes. For more details on L\'evy processes see \cite{sato:1991,applebaum:2004}.

Regular variation in multivariate L\'evy processes, especially characterizing complex tail events, including but not restricted to \eqref{eq:setA},  can happen in a variety of ways. We may observe regular variation for the L\'evy process itself, or the L\'evy measure, and they may have different implications depending on the dependence structure of the L\'evy process. In the following three subsections we investigate this in detail; the proofs of the associated results are provided in  \Cref{sec:proofLevy}. 


\subsection{The L\'evy measure admits regular variation on all subcones}

In the first subsection, we  assume that the L\'evy measure is multivariate regularly  varying on all subcones $\E_d^{(i)}$, $i=1,\ldots,d$ and show that the same is true for the L\'evy process, in fact, they are \emph{tail equivalent} (in a multivariate sense) as we exhibit next. We understand (multivariate) tail equivalence as an extension of \eqref{eq:measuretailequiv} to appropriate sets $A\in \mathcal{B}^{(i)}$. The result can be seen as an extension of \eqref{eq1}  to subcones
(cf. \cite{hult:lindskog:2006b}).

\begin{proposition}[Extending \Cref{prop:mrv:tailsremainsame}] \label{Proposition 6.1}
Let $(\bL(s))_{s\geq 0}$ be a L\'evy process in $\R_+^d$ with L\'evy measure $\Pi\in\MRV(\alpha_i,b_i,\mu_i,\E_d^{(i)})$  for $i=1,\ldots,d$ whose univariate marginal L\'evy measures are tail equivalent.
Moreover assume that $\alpha_i<\alpha_m+\alpha_{i-m}$ for all $m= 1,\ldots, i-1$ and $i = 2,\ldots,d$. For $s>0$ we have then
\beao  \bL(s)\in\MRV(\alpha_i,b_i,s\mu_i,\E_d^{(i)}) \quad  \text {for } \quad i=1,\ldots,d.
\eeao
\end{proposition}

A direct consequence of \Cref{Proposition 6.1} is the tail equivalence of the L\'evy measure of the set $tA$ and the probability measure of the L\'evy process belonging to $tA$, for  Borel sets $A\in\mathcal{B}^{(i)}$ bounded away from $\bCA_d^{(i-1)}$ with $\mu_i(A)>0$ and $\mu_i(\partial A)=0$ such that
\beao
    \P(\bL(s)\in tA)\sim s\,\P(\bL(1)\in tA)\sim s\,\Pi(tA)\sim  \frac{s}{b_i^{\leftarrow}(t)}\, \mu_i(A)
    \quad \text{ as } t\to\infty.
\eeao
 Although the tail equivalence of the L\'evy process and the L\'evy measure holds for a variety of sets, the tail rate  differs depending on which subcone $\E_d^{(i)}$ the set $A$ belongs to. A similar conclusion was shown in  \cite{hult:lindskog:2006b}, but only for sets $A$ with $\mu_1(A)>0$. However, in many situations this is not the case as we see in the following examples.

\begin{example} \label{Example:5.2}
Let $\bL$ be a  compound Poisson process of the form $\bL(s)=\sum_{k=1}^{N(s)}\bZ^{(k)}$ where the jump sizes $(\bZ^{(k)})_{k\in\N}$ are i.i.d. and independent of the Poisson process $(N(s))_{s\geq 0}$ with intensity $\lambda>0$. 
Suppose the jump size $\bZ^{(1)}$ has identical marginals which have $\RV_{-\alpha}$ tail distributions with tail index $-\alpha<0$.
\begin{enumerate}
    \item[(a)]  Let the dependence structure of $\bZ^{(1)}$
    be modelled by a Marshall-Olkin copula with equal parameters as in \Cref{example:Mashall-Olkin}(a), then
    $\bL(s)\in\MRV(\alpha_i,b_i,s\lambda \mu_i,\E_d^{(i)})$ with parameters  $\alpha_i, b_i, \mu_i$   given in \eqref{eq:MO:alphakbkmukeq}.
    \item[(b)] Let the dependence structure of $\bZ^{(1)}$
    be modelled by a Marshall-Olkin copula with proportional parameters
    as in \Cref{example:Mashall-Olkin}(b), then $\bL(s)\in \MRV(\alpha_i^*,b_i^*,s\lambda\mu_i^*,\E_d^{(i)})$
    with parameters $\alpha_i^*, b_i^*, \mu_i^*$  given in \eqref{eq:MO:alphakbkmukpropsize}.
    \item[(c)]  Let the dependence structure of $\bZ^{(1)}$
    be modelled by an ACIG copula as in \Cref{example:Archimedean},  then  $\bL(s)\in \MRV(\alpha,b_{\alpha},s\lambda\mu_1,\E_d^{(1)})$ and $\bL(s)\in\MRV(\alpha\beta,b_{\alpha}^{1/\beta},s\lambda\mu_2,\E_d^{(i)})$ for $i=2,\ldots, d$ with parameters  given in \Cref{ex:cop:arch}.
\end{enumerate}
\end{example}

\begin{example} \label{Example:5.3}
Suppose $L_{\alpha}^{(j)}$, $j=1,2,3$ are i.i.d. L\'evy processes in $\R_+$ with L\'evy measure $\Pi_\alpha\in \MRV(\alpha,b_{\alpha},\mu_\alpha,(0,\infty))$, 
$L_{\beta}^{(j)}$, $j=1,2$ are i.i.d. L\'evy processes in $\R_+$ with L\'evy measure $\Pi_\beta\in\MRV(\beta,b_{\beta},\mu_\beta,(0,\infty))$ and $L_{\gamma}$ is a L\'evy process in $\R_+$ with L\'evy measure $\Pi_\gamma\in\MRV(\gamma,b_{\gamma},\mu_\gamma,(0,\infty))$. A typical example for $L_\alpha$
is an $\alpha$-stable L\'evy process with L\'evy measure $\Pi_{\alpha}(\mathrm dx)=\alpha x^{-\alpha-1}\bone_{(0,\infty)}(x)\, \mathrm dx$. Furthermore, assume all processes are independent and $\alpha<\beta<\gamma<2\alpha$. Then the 3-dimensional L\'evy process
\beao
    \bL(s)=(L_{\alpha}^{(1)}(s)+L_{\beta}^{(1)}(s)+L_{\gamma}(s),L_{\alpha}^{(2)}(s)+L_{\beta}^{(1)}(s)+L_{\gamma}(s),L_{\alpha}^{(3)}(s)+L_{\beta}^{(2)}(s)+L_{\gamma}(s)),
\eeao
has L\'evy measure
\beao
    \Pi(A)=\sum_{j=1}^3\Pi_\alpha(A_j)
        +\Pi_\beta(A_1\cap A_2)+\Pi_\beta(A_3)
        +\Pi_\gamma(A_1\cap A_2\cap A_3),
\eeao
where $A_1=\{z\in\R_+:(z,0,0)\in A\}$,
$A_2=\{z\in\R_+:(0,z,0)\in A\}$ and
$A_3=\{z\in\R_+:(0,0,z)\in A\}$. Of course, $\Pi$ satisfies the assumptions of \Cref{Proposition 6.1} with $(\alpha_1,\alpha_2,\alpha_3)=(\alpha,\beta,\gamma)$, $(b_1,b_2,b_3)=(b_{\alpha},b_{\beta},b_{\gamma})$ and 
\beao
    \mu_1(A)=\sum_{j=1}^3\mu_\alpha(A_j),\quad \mu_2(A)=\mu_\beta(A_1\cap A_2)
      \quad  \text{ and } \quad
     \mu_3(A)=\mu_\gamma(A_1\cap A_2\cap A_3).
\eeao
Finally, $\bL(s)\in\MRV(\alpha_i,b_i,s\mu_i,\E_d^{(i)})$ for $i=1,2,3.$
\end{example}

In each case for Examples \ref{Example:5.2} and \ref{Example:5.3},  $\bL(s)$ is MRV on $\E_d^{(1)}$ and $\E_d^{(2)}$ but with different indices and hence, $\mu_1(\E_d^{(2)})=0$. This implies that the components of $\bL(s)$ are asymptotically tail independent. 
In the special case where the components of  $\bL(s)$ are (strongly) dependent, the next result follows directly from \Cref{Proposition 6.1}.

\begin{corollary}[Extending \Cref{cor:asymdep}: Dependent case] \label{corollary2}
Let $(\bL(s))_{s\geq 0}$ be a L\'evy process in $\R_+^d$ with  L\'evy measure $\Pi\in\MRV(\alpha_i,b_i,\mu_i,\E_d^{(i)})$  for $i=1,\ldots,d$
whose univariate marginal L\'evy measures are tail equivalent. Moreover,  $(\alpha_i,b_i,\mu_i,\E_d^{(i)}) = (\alpha,b,\mu,\E_d^{(i)})$ for $i=1,\ldots, i^*$ for some $ i^*\le d$.
 Then for $s>0$ we have
\beao
                \bL(s)\in\MRV(\alpha,b,s\mu,\E_d^{(i)}) \quad  \text {for } \quad i=1,\ldots,  i^*.
\eeao
\end{corollary}

\begin{example}  $\mbox{}$
\begin{enumerate}
    \item[(a)] Completely dependent case: Let $L_1=(L_1(s))_{s\geq 0}$ be a L\'evy process in $\R_+$ with univariate marginal L\'evy measure $\Pi_\alpha\in \MRV(\alpha,b_{\alpha},\mu_\alpha,(0,\infty))$ and $\bL=(L_1,\ldots,L_1)$. Then the L\'evy measure of $\bL$ is given by
    \beao
        \Pi(A)=\min_{j\in S}\Pi_\alpha((x_j,\infty))
    \eeao
    for a rectangular set $A= \{\bz \in \R_+^d: z_j>x_j \;\forall\; j\in S\}$  with
 $S\subseteq \mathbb{I}$ and  $x_{j}>0$ for $j\in S$.
   In this case, we are in the setting of \Cref{cor:asymdep} with $i^*=d$ and $\mu(A)=\min_{j\in S}x_j^{-\alpha}$ for
   a rectangular set $A$ as above.
   
    More generally,  if the marginal 
    tail L\'evy measures are not necessarily identical {but are completely tail equivalent} satisfying $ \Pi_j\in \MRV(\alpha,b_{\alpha},\mu_\alpha,(0,\infty))$
for $j=1,\ldots,d$
    and 
    \beao
        \Pi(A)=\min_{j\in S}\Pi_j(x_j,\infty)
    \eeao
    then $\Pi\in \MRV(\alpha,b_{\alpha},\mu,\E_d^{(i)})$ for $i=1,\ldots,d$ as well and the assumptions of \Cref{corollary2} are satisfied. Indeed, this is a L\'evy measure, it is constructed by the complete dependence L\'evy copula (cf. \cite{Kallsen:Tankov:2006}).
    \item[(b)]  Suppose $L_j$, $j=1,\ldots,d$ are L\'evy processes in $\R_+$ with univariate {completely tail equivalent} marginal L\'evy measures $\Pi_j\in \MRV(\alpha,b_{\alpha},\mu_\alpha,(0,\infty))$
    and $\bL(s)=(L_1(s),\ldots,L_d(s))$ is a $d$-dimensional
    L\'evy process with L\'evy measure
    \beao
        \Pi(A)=\left(\sum_{j\in S}\left(\Pi_j((x_j,\infty))\right)^{-\theta}\right)^{\frac{1}{\theta}}
    \eeao
    for some $\theta>0$, where $A= \{\bz \in \R_+^d: z_j>x_j \;\forall\; j\in S\}$ is a rectangular set with
 $S\subseteq \mathbb{I}$  and  $x_{j}>0$ for $j\in S$.
    This  L\'evy measure is constructed using the Clayton L\'evy copula (cf. \cite{Kallsen:Tankov:2006}). Let the measure $\mu$ { on $\mathcal{B}^{(1)}$} be defined as
    \beao
        \mu(A)=\left(\sum_{j\in S}x_j^{-\alpha\theta}\right)^{-\frac{1}{\theta}}
    \eeao
    for a rectangular set $A$ as above.  Then $\Pi\in
    \MRV(\alpha,b_{\alpha},\mu,\E_d^{(i)})$ for $i=1,\ldots,d$
    and hence, due to \Cref{corollary2}, 
    $\bL(s)\in
    \MRV(\alpha,b_{\alpha},s\mu,\E_d^{(i)})$ for $i=1,\ldots,d$ as well. 
    \item[(c)]{Another (dependent) example of a L\'evy process can be constructed by a compound Poisson process where the jumps sizes have the distribution $F$ as in \Cref{ex:cop:dep}.}
\end{enumerate}
\end{example}

\begin{remark}
Regular variation of the L\'evy measure on different subcones $\E_d^{(i)}$ can be related to regular variation of the L\'evy copula and Pareto L\'evy copula, respectively on these different subcones; cf. \cite{kluppelberg:resnick:2008,eder:kluppelberg:2012} for classical regular variation of such L\'evy measure on $\E_d^{(1)}$.  This work is under investigation by the authors.
\end{remark}

\subsection{The L\'evy process is asymptotically tail independent}
In \Cref{Proposition 6.1} and subsequently \Cref{corollary2}, the underlying L\'evy measure admits regular variation on all relevant subcones; but this may not necessarily be the case in general.
The next result includes the cases where the L\'evy measure is adapted multivariate regularly varying; i.e., MRV need not exist in all the relevant subcones.

\begin{proposition}[Extending \Cref{prop:Delta1}] \label{Proposition 6.3}
Let $(\bL(s))_{s\geq 0}$ be a L\'evy process in $\R_+^d$ with
 L\'evy measure $\Pi$ such that $\{\Pi \in \MRV^*(\alpha_i,b_i,\mu_i,\E_d^{(i)}), i=1,\ldots,d; \Delta=1\}$ and $\Pi$ has tail equivalent univariate marginal L\'evy measures in $\RV_{-\alpha}$.  Then for $s>0$ we have
\begin{eqnarray*}
    \bL(s)\in \MRV(i\alpha_1,b_1^{1/i},s^i \mu_i^L,\E_d^{(i)})
    \quad \text {for } \quad i=1,\ldots,  d.
\end{eqnarray*}
with
\begin{align}\label{eq:limit:muiplusdelta1new}
    \mu_{i}^{L}\Big(\bigcap_{j\in S}\{\bz\in \E_d^{(i)}:z_j>x_j\}\Big)= \prod_{j\in S}
    \mu_1\left(\{\bz\in \E_d^{(i)}:z_j>x_j\}\right).
    \end{align}
    for $S\subseteq \mathbb{I}$ with $|S|=i$, $x_j>0$ for $j\in S$  and $\mu_{i}^{L}(\E_d^{(i+1)})=0$.
\end{proposition}
Interestingly, for rectangular sets $A\in\mathcal{R}^{(i)}$ as in \eqref{eq:setA} with $|S|=i$, now  we observe that
\begin{eqnarray} \label{eq;1*}
    \P(\bL(s)\in tA)\sim s^{i}\,\P(\bL(1)\in tA) \quad \text{ as }t\to\infty.
\end{eqnarray}
Hence, the linearity property of $\P(\bL(s)\in tA)\sim s\,\P(\bL(1)\in tA)$ as $t\to\infty$, which we had noticed in the dependent cases of 
\Cref{Proposition 6.1} and \Cref{corollary2} respectively,  vanishes here making this an unusual phenomenon for L\'evy processes. Moreover, although $\Pi$ is MRV on $\E_d^{(1)}$, for sets $A\in\mathcal{R}^{(2)}$ the tail measures
$\Pi(tA)$ and $\P(\bL(1)\in tA)$ are not tail equivalent anymore, in contrast to the common wisdom
for regular variation of L\'evy processes on $\E_d^{(1)}$.

\begin{example} $\mbox{}$
\begin{enumerate}
\item[(a)] Suppose the marginal L\'evy processes $L_1,\ldots,L_d$ of $\bL=(L_1,\ldots,L_d)$
are independent with tail equivalent univariate marginal L\'evy measures $ \Pi_j $ which are regularly varying with tail index $-\alpha<0$. Then the L\'evy measure of $\bL$ is
\beao
    \Pi(A)=\sum_{j=1}^d\Pi_j(A_j)
\eeao
for $A_j=\{z\in\R_+:(0,\ldots,0,z,0,\ldots,0)\in A\}$, where $z$ appears in the $j$-th coordinate. This L\'evy measure 
has  mass only on the co-ordinate axes. 
Hence, the assumptions of \Cref{Proposition 6.3} are satisfied and it can be applied to show MRV of $\bL$ on various subcones. {In particular, it satisfies \eqref{eq;1*}.}
\item[(b)]  A compound Poisson process as defined in \Cref{Example:5.2} with  jumps sizes $(\bZ^{(k)})_{k\in\N}$ as in \Cref{ex:mrvsum:delta1} would also satisfy the assumptions of \Cref{Proposition 6.3}, providing an example for the same.
\end{enumerate}
\end{example}

\subsection{The L\'evy measure is asymptotically tail independent}
The next proposition covers the case of a compound Poisson process, as a special L\'evy process, where the marginal distribution of the jump sizes 
are independent as well. The observed phenomena is again different from \Cref{Proposition 6.3}, which covers a compound Poisson process with independent marginal L\'evy processes. 


\begin{proposition}[Extending \Cref{prop:ind}] \label{Corollary:2}
Let $(\bL(s))_{s\geq 0}$ be a L\'evy process in $\R_+^d$ with L\'evy measure $\Pi\in\MRV(\alpha_i,b_i,\mu_i,\E_d^{(i)})$  for $i=1,\ldots,d$
where $b_i(t)=(b_1(t))^{1/i}$, $b_1(t)\in \RV_{1/\alpha}$ and
    \begin{eqnarray*}
            \mu_i\left(\bigcap_{j\in S}\{\bz\in \E_d^{(i)}:z_j>x_j\}\right)=\prod_{j\in S}\kappa_j x_j^{\alpha}
           \end{eqnarray*}
         for $S\subseteq \mathbb{I}$ with $|S|=i$, $x_j>0$ for $j\in S$   and $\mu_{i}(\E_d^{(i+1)})=0$, $i=1,\ldots, d$,
         and $\Pi$ has tail equivalent univariate marginal L\'evy measures. Let $(N^*(s))_{s\geq 0}$ denote a Poisson
         process with intensity 1.
         Then for $s>0$ we have
         \beao
                \bL(s)\in\MRV(i\alpha,b_1^{1/i},\E(N^*(s)^{i})\mu_i,\E_d^{(i)}) \quad  \text {for } \quad i=1,\ldots,d.
         \eeao
\end{proposition}
As a consequence of \Cref{Corollary:2}, for any rectangular set $A\in\mathcal{R}^{(i)}$ as in \eqref{eq:setA} with $|S|=i$
we obtain
\begin{eqnarray*}
    \P(\bL(s)\in tA)\sim \frac{\E(N^*(s)^{i})}{\E(N^*(1)^{i})}\P(\bL(1)\in tA) \quad \text{ as }t\to\infty,
\end{eqnarray*}
where $\E(N^*(s)^{i})$ is a polynomial of order $i$ in $s$, and $s^i\leq \E(N^*(s)^{i})$.

\begin{remark} $\mbox{}$
\begin{enumerate}
    \item[(a)] We can verify that indeed the result in \Cref{Corollary:2} is in accordance with  \Cref{mainTheorem}. From \Cref{Corollary:2} we get that for the  i.i.d.  random vectors
\begin{eqnarray*}
    \bL(1) \text{ and } \bL(2)-\bL(1)\in\MRV(i\alpha,b_1^{1/i},\E(N^*(1)^{i})\mu_i,\E_d^{(i)}) \quad  \text {for } \quad i=1,\ldots,d.
\end{eqnarray*}
Applying \Cref{mainTheorem} gives
\begin{eqnarray*}
    \bL(2)= \bL(1)+ [\bL(2)-\bL(1)]\in\MRV(i\alpha,b_1^{1/i},\mu_i^{\oplus},\E_d^{(i)}) \quad  \text {for } \quad i=1,\ldots,d
\end{eqnarray*}
with
\begin{eqnarray*}
    \mu_i^{\oplus}&=&\sum_{m=0}^i\binom{i}{m}
        \E(N^*(1)^m)\E([N^*(2)-N^*(1)]^{i-m})\mu_i\\
        &=&\E\left(\sum_{m=0}^i\binom{i}{m}
        N^*(1)^m[N^*(2)-N^*(1)]^{i-m}\right)\mu_i\\
        &=&\E(N^*(1)+[N^*(2)-N^*(1)])^{i}\mu_i
        =\E(N^*(2)^{i})\mu_i
\end{eqnarray*}
which is also a consequence of \Cref{Corollary:2}.

\item[(b)]
Suppose $(\bL(s))_{s\geq 0}$ is a compound Poisson process with L\'evy measure \linebreak $\Pi\in\MRV(\alpha_i,b_i,K\mu_i,\E_d^{(i)})$  for $i=1,\ldots,d$ with $\alpha_i,b_i,$ and $\mu_i$ as in \Cref{Corollary:2}
and $K>0$ is some positive constant. Furthermore, suppose the marginal L\'evy measures of $\Pi$ are tail-equivalent. Let $(\widetilde \bL(s))_{s\geq }$
be another compound Poisson process with L\'evy measure $\widetilde \Pi
=\Pi/K$. Then $\widetilde \Pi\in\MRV(\alpha_i,b_i,\mu_i,\E_d^{(i)})$ for $i=1,\ldots,d$
and due to \Cref{prop:ind}, we have
\beao
                \widetilde \bL(s)\in\MRV(i\alpha,b_1^{1/i},\E(N^*(s)^{i})\mu_i,\E_d^{(i)}) \quad  \text {for } \quad i=1,\ldots,d.
         \eeao
But $\bL(s)\eqd \widetilde \bL(Ks)$ and hence, we have
\beao
                 \bL(s)\in\MRV(i\alpha,b_1^{1/i},\E(N^*(Ks)^{i})\mu_i,\E_d^{(i)}) \quad  \text {for } \quad i=1,\ldots,d.
\eeao
\end{enumerate}
\end{remark}

\begin{example}
Consider the compound Poisson process
\beao
    \bL(s)=\sum_{k=1}^{N^*(s)}(Z_1^{(k)},\ldots,Z_d^{(k)})
\eeao
where $(N^*(s))_{s\geq 0}$ is a Poisson process with intensity $1$, which is independent of the i.i.d. sequence of jump sizes $(Z_1^{(k)},\ldots,Z_d^{(k)})_{k\in\N}$. Suppose 
$(Z_1^{(k)})_{k\in\N}$, \ldots, $(Z_d^{(k})_{k\in\N}$
are as well independent of each other with tail equivalent marginal distributions $F_{j}$ and $\ov F_j\in\mathcal{RV}_{-\alpha}$. Then 
for a rectangular set $A= \{\bz \in \R_+^d: z_j>x_j \;\forall\; j\in S\}$  with
 $S\subseteq \mathbb{I}$ and  $x_{j}>0$ for $j\in S$ we have
\beao
    \Pi(A)=\prod_{j\in S}^d\ov F_j(x_j).
\eeao
Thus, the assumptions  of \Cref{Corollary:2} are again satisfied {and hence, can be applied here}.
\end{example}

\bibliographystyle{imsart-nameyear}
\bibliography{bibfilenew}



\appendix

\section{Proofs of the results in Section~\ref{sec:prelim}} \label{sec:proofprelim}

\begin{proof}[Proof of \Cref{lem:RiisPi}]
The set $(1,\infty)^d\in \cR^{(i)},$ and hence, it is non-empty. Now, let $A$ and  $B$ be two arbitrary sets in $\cR^{(i)}$. Then for some $m,n  \ge i$, with $x_{j}>0, j\in \{k_1,\ldots,k_m\}=:S_1 \subseteq \mathbb{I}$ and $y_{j}>0, j\in \{\ell_1,\ldots,\ell_n\}  =:S_2\subseteq \mathbb{I}$ we have
\begin{align*}
A= \left\{\bz\in\R_+^d: z_{j}>x_j\;\forall\, j\in  S_1 \right\}, \quad  B= \left\{\bz\in\R_+^d: z_{j}>y_j\;\forall\, j\in S_2 \right\}.
\end{align*}
For  $j\in S^*:= S_1 \cup S_2$, define $$w_j = \begin{cases} x_j,  & \text{ if } j \in S_1 \cap {S_2}^{c},\\
 y_j & \text{ if } j \in {S_1}^{c} \cap {S_2},\\
  \max\{x_j,y_j\} &\text{ if }  j \in S_1 \cap {S_2}.\end{cases}$$
Thus, $A\cap B = \left\{\bz\in\R_+^d: z_{j}>w_j\;\forall\; j\in  S^* \right\} $ where $|S^*| \ge \max(m,n)\ge i$ and $w_j>0\; \forall\, j \in S^*$. Hence, $A\cap B \in \cR^{(i)}$ and $\cR^{(i)}$ is a $\pi$-system. It can also be checked that  $\sigma(\cR^{(i)})=\cB^{(i)}$.
\end{proof}

\begin{proof}[Proof of \Cref{prop:rectsetsforM}]
\underline{\eqref{Mcon:mut} $\Rightarrow$  \eqref{Mcon:mutAy}:} Using \cite[Theorem 2.1]{lindskog:resnick:roy:2014},  if \eqref{Mcon:mut} holds, then any set of the form $A\in \cR^{(i)}$ with $\mu(\partial A)=0$ is clearly bounded away from $\bCA_i$ and belongs to the $\sigma$-algebra $\cB^{(i)}$ as defined in \eqref{eq:Bi}. Hence, \eqref{Mcon:mutAy} holds as $t\to\infty$.\\
 \underline{\eqref{Mcon:mutAy} $\Rightarrow$  \eqref{Mcon:mut}:}  Now assume  \eqref{Mcon:mutAy}  holds for all $\mu$-continuity sets $A\in \cR^{(i)}$.  Denote by $M$, the collection $$M=\{\mu_t: t>0\}\subset \M(\E_d^{(i)}).$$ For any $r>0$, let $\nu^{(r)}$ be the restriction of   $\nu\in \M(\E_d^{(i)})$ to $\E_d^{(i)}\setminus \bCA_i^{(r)}$ where $\bCA_i^{(r)}=\{\bz\in \R_+^d: d(\bz,\bCA_i)<r\}$. Let $M^{(r)} := \{\nu^{(r)}: \nu \in M\}$ and let $\{r_{\ell}\}$ be a sequence $r_{\ell}\downarrow 0$ as $\ell\to \infty$. Note that $M^{(r)} \subset \M(\E_d^{(i)}\setminus \bCA_i^{(r)})$ which is a class of finite Borel measures.

 Denote by $\mathcal{C}_{i}^{(r)}=$ all real-valued, bounded continuous functions $f$ on ${\E_d^{(i)}}$ which vanishes on $\bCA_i^{(r)}$. Fix $\ell\ge 1$ and pick any $f\in\mathcal{C}_{i}^{(r_{\ell})} $ which is uniformly continuous. Then by definition, the support of $f$ lies on a finite union of rectangular sets  given by $$A= \bigcup_{\genfrac{}{}{0pt}{}{S\subset\mathbb{I}}{|S|=i}} A_S,$$ where $A_S= \{\bz\in \R_+^d: z_j>r^* \;\forall\; j\in S\} \in \cR^{(i)}$ for some $0< r^*< r_{\ell}$. W.l.o.g. the sets $A_S$ can be assumed  to be $\mu$-continuity sets  using \cite[Lemma 2.5]{lindskog:resnick:roy:2014}. Now, using \eqref{Mcon:mutAy} we have convergence on the sets $A_S$ and  therefore,
\begin{align*}
\sup_{\nu\in M^{(r_{\ell})}} \nu(f)= \sup_t \mu_t^{(r_\ell)}(f) & \le \sup\limits_{\bz\in \R_+^d} f(\bz) \sup\limits_{t} \mu_t(A)\\ &\le \sup\limits_{\bz\in \R_+^d} f(\bz) \sum_{\genfrac{}{}{0pt}{}{S\subset\mathbb{I}}{|S|=i}}\sup\limits_{t} \mu_t^{(r_{\ell})}(A_S)<\infty.
\end{align*}
Hence, for any sequence of measures $\{\nu_n\}_{n\ge1} \in M^{(r_{\ell})}$, the sequence $\nu_n(f)$ has a convergent subsequence. Since this is true for any uniformly continuous $f\in \mathcal{C}_{i}^{(r_{\ell})}$, it is true for any   $f\in \mathcal{C}_{i,K}^{(r_{\ell})}$, which are compactly supported functions in $\mathcal{C}_{i}^{(r_{\ell})}$. Since $\mathcal{C}_{i,K}^{(r_{\ell})}$ is separable, using a countable dense collection $\{f_j\}_{j\ge 1}\in \mathcal{C}_{i,K}^{(r_{\ell})}$, and a diagonal argument we can show that any sequence of measures $\{\nu_n\}_{n\ge1} \in M^{(r_{\ell})}$ has a  subsequence $\{\nu_{n_k}\}$ such that $\lim_{n_k\to\infty}\nu_{n_k}(g)=\nu(g)$ for any $g\in \mathcal{C}_{i,K}^{(r_{\ell})}$ and hence for all uniformly continuous functions $f\in \mathcal{C}_{i}^{(r_{\ell})}$ (using a sequence $g_n \to f$ where $g_n\in \mathcal{C}_{i,K}^{(r_{\ell})}$). Thus, $M^{(r_{\ell})}=\{\mu_t^{(r_{\ell})}:t>0\}$ is relatively compact; cf. \citep[(3.16), p. 51)]{resnickbook:2007}; and this holds for a sequence $\{r_{\ell}\}$ where $r_{\ell}\downarrow 0$. Also $M^{(r_{\ell})} \subset M$. Hence, by \cite[Theorem 2.4]{lindskog:resnick:roy:2014} we have $M$ is relatively compact.

 Suppose as $t\to\infty$, $\mu_t$ has two different sequential limits $\mu_1$ and $\mu_2$, then by assumption they clearly agree on all sets $A\in \cR^{(i)}$. By Lemma \ref{lem:RiisPi} such rectangular sets form a  $\pi$-system generating the $\sigma$-algebra $\cB^{(i)}$. Hence, $\mu_1=\mu_2=\mu$ on $\E_d^{(i)}.$
\end{proof}

\begin{proof}[Proof of \Cref{Lemma_aux}]
Let $A= \{\bz \in \R_+^d: z_j>x_j\;\forall\, j\in S\}$ where $S\subseteq \mathbb{I}$, $|S|\ge i$, $x_j>0$ for $j\in S$ and $\mu_i(\partial A)=0$. Furthermore, let $0<\varepsilon<\min_{j\in S}x_j$ and define the sets
\begin{eqnarray*}
    A_{\varepsilon}^+&:=&\{\bz \in \R_+^d: z_j>x_j-\varepsilon \; \forall\; j\in S\},\\
    A_{\varepsilon}^-&:=&\{\bz \in \R_+^d: z_j>x_j+\varepsilon \; \forall\; j\in S\},\\
    N_{\varepsilon}&:=&\{\bz \in \R_+^d: |z_j|\leq \varepsilon \;\forall\; j\in S\}.
\end{eqnarray*}
Suppose w.l.o.g. $\mu_i(\partial A_{\varepsilon}^+)=\mu_i(\partial A_{\varepsilon}^-)=0$ (otherwise choose $\varepsilon$ appropriate).
On the one hand,
\begin{eqnarray*}
    \P(\bX+\bY\in tA)
    &\leq &\P(\bX\in tA_{\varepsilon}^+)
    +\P(\bY\in tN_{\varepsilon}^c)\\
    &\leq &\P(\bX\in tA_{\varepsilon}^+)
    +\P(\|\bY\|_{\infty}>\varepsilon t).
\end{eqnarray*}
Hence,
\begin{eqnarray}
    \limsup_{t\to\infty}b_i^{\leftarrow}(t)
    \P(\bX+\bY\in tA)
    &\leq &\limsup_{t\to\infty}b_i^{\leftarrow}(t)
    \P(\bX\in tA_{\varepsilon}^+)
    +\limsup_{t\to\infty}b_i^{\leftarrow}(t)
    \P(\|\bY\|_{\infty}>\varepsilon t) \nonumber\\
    &\leq& \mu_i(A_{\varepsilon}^+)+
        \limsup_{t\to\infty}b_i^{\leftarrow}(t)(\gamma t)^{-(\alpha_i+\gamma)}\E\|\bY\|^{\alpha_i+\gamma} \nonumber\\
    &=&\mu_i(A_{\varepsilon}^+)\downarrow \mu_i(A) \quad \text{ as } \varepsilon\downarrow 0, \label{eq:lem2.9:sup}
\end{eqnarray}
since $\mu_i(\partial A)=0$. On the other hand,
\begin{eqnarray*}
    \P(\bX+\bY\in tA)
    &\geq &\P(\bX+\bY\in tA, \bY\in tN_{\varepsilon})\\
    &\geq &\P(\bX\in tA_{\varepsilon}^-, \bY\in tN_{\varepsilon})\\
    &=&\P(\bX\in tA_{\varepsilon}^-)\P(\bY\in tN_{\varepsilon}). 
\end{eqnarray*}
Therefore,
\begin{eqnarray}
    \liminf_{t\to\infty}b_i^{\leftarrow}(t)
    \P(\bX+\bY\in tA)
    &\geq& \limsup_{t\to\infty}b_i^{\leftarrow}(t)\,\P(\bX\in tA_{\varepsilon}^-)\P(\bY\in tN_{\varepsilon}) \nonumber\\
    &=& \mu_i( A_{\varepsilon}^-)\uparrow \mu_i(A) \quad \text{ as } \varepsilon\downarrow 0, \label{eq:lem2.9:inf}
\end{eqnarray}
since $\mu_i(\partial A)=0$. Thus \eqref{eq:lem2.9:sup} and \eqref{eq:lem2.9:inf} imply that
\begin{eqnarray*}
    \lim_{t\to\infty}b_i^{\leftarrow}(t)
    \P(\bX+\bY\in tA)=\mu_i(A),
\end{eqnarray*}
and using \Cref{prop:rectsetsforM} we can conclude the statement.
\end{proof}

\section{Proof of Theorem~\ref{mainTheorem}} \label{sec:proofmaintheorem}

The following auxiliary lemmas are used to prove \Cref{mainTheorem}.

\begin{lemma} \label{Lemma 4.2}
Let the assumptions of \Cref{mainTheorem} hold. Then for any
$m=0,\ldots,i-1$:
\begin{align}
    \lim_{t\to\infty}b_{I(i)}^{(1)\leftarrow}(t)b_{i-I(i)}^{(2)\leftarrow}(t)\,\P(Z_{(m+1)}^{(1)}>t)\,\P(Z_{(i-m)}^{(2)}>t)=& 0 \label{appeq:B1}\\ \intertext{ and }
   { \lim_{t\to\infty}b_{I(i)}^{(1)\leftarrow}(t)b_{i-I(i)}^{(2)\leftarrow}(t)\,\P(Z_{(m+1)}^{(2)}>t)\,\P(Z_{(i-m)}^{(1)}>t)}=& 0. \label{appeq:B2}
\end{align}
\end{lemma}
\begin{proof}

Let $\Gamma^{(k)}:=\arg\max_i\{\alpha_i^{(k)}<\infty\}$, $k=1, 2$. By definition, for $i\le \Gamma^{(k)}$, we have $\P(Z_{(i)}^{(k)}>t) =O( 1/b_i^{(k)\leftarrow}(t))$; and for $i>\Gamma^{(k)}, \P(Z_{(i)}^{(k)}>t) = o(1/b_i^{(k)\leftarrow}(t))$. Hence, to prove \eqref{appeq:B1}, we need only to show that 
\beao
    a_m:=\limsup_{t\to\infty}\frac{b_{I(i)}^{(1)\leftarrow}(t)b_{i-I(i)}^{(2)\leftarrow}(t)}{b_{m+1}^{(1)\leftarrow}(t)b_{i-m}^{(2)\leftarrow}(t)}=0, \quad m=0,\ldots,i-1.
\eeao
For $m=0$ we have
$$a_0=c_0^{I(i)}\lim_{t\to\infty}1/b^{(1)\leftarrow}_1(t)=0.$$
Let $ m\in\{1,\ldots,i-1\}$.
Note that
\beao
        a_{m}
        =c_{m}^{I(i)}\limsup_{t\to\infty}\frac{b_{m}^{(1)\leftarrow}(t)}{b_{m+1}^{(1)\leftarrow}(t)}.
\eeao
Since $\limsup_{t\to\infty}b_{m}^{(1)\leftarrow}(t)/b_{m+1}^{(1)\leftarrow}(t)<\infty$, the last equality implies that  $a_m=0$ is only possible if either $c_m^{I(i)}=0$ or $\limsup_{t\to\infty}b_{m}^{(1)\leftarrow}(t)/b_{m+1}^{(1)\leftarrow}(t)=0$, which holds true by the assumptions in \Cref{mainTheorem}.
The proof of \eqref{appeq:B2} is analogous.
\end{proof}

\begin{lemma} \label{Lemma 3.4}
Let the assumptions of \Cref{mainTheorem} hold and
%
 $A= \{\bz \in \R_+^d: z_j>x_j \;\forall\; j\in S\}$ be a rectangular set with
 $S\subseteq \mathbb{I}$, $|S|= i$, $x_{j}>0$ for $j\in S$ and $\mu_i^{\oplus}(\partial A)=0$. Then
\beam\label{eq:lemma2}
   \lim_{t\to\infty}b_{I(i)}^{(1)\leftarrow}(t)b_{i-I(i)}^{(2)\leftarrow}(t) \P\left(\bZ^{(1)}+\bZ^{(2)}\in tA\right)
        =\mu_i^{\oplus}(A).
\eeam
\end{lemma}

\begin{proof} $\mbox{}$\\
 \textbf{Step 1.} First, we derive an upper bound for the left hand side of \eqref{eq:lemma2}.  Let $0<\epsilon<1$
 such that $\epsilon x_j<(1-\epsilon)x_l$ for all $j,l\in S$.
 Then
\beam \label{A6}
    \P\left(\bZ^{(1)}+\bZ^{(2)}\in tA\right)&=&\Bigg[\sum_{\genfrac{}{}{0pt}{}{J_2\subseteq S\cup \{\emptyset\}}{J_1=S\backslash J_2}}+\sum_{\genfrac{}{}{0pt}{}{J_2\subseteq  S\cup\{\emptyset\}}{J_1 \subsetneq S\backslash J_2\cup \{\emptyset\}}}\Bigg]\P\left(\bigcap_{j\in S}\left\{Z_j^{(1)}+Z_j^{(2)}>tx_j\right\}
                 \right. \nonumber\\
        &&\quad\quad \cap\bigcap_{j\in J_1}\left\{Z_j^{(1)}\leq t\epsilon x_j\right\}\cap\bigcap_{j\in S\backslash J_1}\left\{Z_j^{(1)}> {t (1-\epsilon)} x_j\right\} \nonumber\\
        &&\quad\quad\left. \cap\bigcap_{j\in J_2}\left\{Z_j^{(2)}\leq t\epsilon x_j\right\}\cap\bigcap_{j\in S\backslash J_2}\left\{Z_j^{(2)}> {t(1-\epsilon)} x_j\right\} \right) \nonumber\\
        &&=:M_{1}(t,\epsilon)+M_{2}(t,\epsilon).
\eeam
Note that in case $J_1\cap J_2 \not=\emptyset$ these probabilities are zero since it results in computing probabilities of empty sets.
Next, we  find an upper bound for $M_{1}(t,\epsilon)$. Note that
\beao
    M_{1}(t,\epsilon)&\leq&\sum_{\genfrac{}{}{0pt}{}{J_2\subseteq S \cup\{\emptyset\}}{J_1=S\backslash J_2}}
        \P\left(\bigcap_{j\in S\setminus J_1}\left\{Z_j^{(1)}> t(1-\epsilon) x_j\right\}\cap\bigcap_{j\in S \setminus J_2}\left\{Z_j^{(2)}> t(1-\epsilon) x_j\right\}\right)\\
        &=&\sum_{J_2\subseteq S\cup\{\emptyset\} }
       \P\left(\bigcap_{j\in J_2}\left\{Z_j^{(1)}> t(1-\epsilon) x_j\right\}\right)
        \P\left(\bigcap_{j\in S\backslash J_2}\left\{Z_j^{(2)}> t(1-\epsilon) x_j\right\}\right).
\eeao
Let $A_{\epsilon}:= \{\bz \in \R_+^d: z_j>(1-\epsilon)x_j \,\forall\, j\in S\}$ and choose $\epsilon>0$ such that $\mu_i^{\oplus}(\partial A_{\epsilon})=0$.
Then
\beam \label{A5}
    \limsup_{t\to\infty}b_{I(i)}^{(1)\leftarrow}(t)b_{i-I(i)}^{(2)\leftarrow}(t)M_{1}(t,\epsilon)
    \leq \mu_i^{\oplus}(A_{\epsilon}). 
\eeam
Define $x^*:=\min_{j\in S} x_j$. Following a similar argument for $M_{2}(t,\epsilon)$ we get the upper bound
\beao
    M_{2}(t,\epsilon)&\leq& \sum_{\genfrac{}{}{0pt}{}{J_2\subseteq S\cup\{\emptyset\}}{ J_1\subsetneq S\backslash J_2\cup\{\emptyset\}}}
        \P\Big(\bigcap_{j\in J_2\cup S\backslash J_1}\left\{Z_j^{(1)}> t(1-\epsilon) x_j\right\}\Big)\P\Big(\bigcap_{j\in  J_1\cup S\backslash J_2}\left\{Z_j^{(2)}> t(1-\epsilon) x_j\right\}\Big)\\
    &\leq& \sum_{\genfrac{}{}{0pt}{}{J_2\subseteq S\cup\{\emptyset\}}{ J_1\subsetneq S\backslash J_2\cup\{\emptyset\}}}
        \P\left(Z_{(|J_2\cup S\backslash (J_1\cup J_2)|)}^{(1)}> t(1-\epsilon) x^*\right)\,\P\left(Z_{(|S\backslash J_2|)}^{(2)}> t(1-\epsilon) x^*\right)\\
    &\leq &\sum_{m=0}^{i-1}\sum_{l=1}^{i-m}\sum_{\genfrac{}{}{0pt}{}{J_2\subseteq S\cup\{\emptyset\} }{ |J_2|=m }}\sum_{\genfrac{}{}{0pt}{}{J_1\subsetneq S\backslash J_2\cup\{\emptyset\}}{|S\backslash ( J_1 \cup J_2)|=l}}
        \P\left(Z_{(m+l)}^{(1)}> t(1-\epsilon) x^*\right)\,\P\left(Z_{(i-m)}^{(2)}> t(1-\epsilon) x^*\right).
\eeao
Finally, an application of \Cref{Lemma 4.2} yields
\beam
    \lefteqn{\limsup_{t\to\infty}b_{I(i)}^{(1)\leftarrow}(t)b_{i-I(i)}^{(2)\leftarrow}(t) M_{2}(t,\epsilon)}\nonumber\\
    && \quad\quad\leq \sum_{m=0}^{i-1}\sum_{l=1}^{i-m}\sum_{\genfrac{}{}{0pt}{}{J_2\subseteq S\cup\{\emptyset\}}{ |J_2|=m }}\sum_{\genfrac{}{}{0pt}{}{J_1\subsetneq S\backslash J_2 \cup\{\emptyset\}}{|S\backslash ( J_1 \cup J_2)|=l}}{\limsup_{t\to\infty} }\, b_{I(i)}^{(1) \leftarrow}(t)b_{i-I(i)}^{(2)\leftarrow}(t) \nonumber\\
    && \quad\quad\quad\quad\quad \P\left(Z_{(m+1)}^{(1)}> t(1-\epsilon) x^*\right)\,\P\left(Z_{(i-m)}^{(2)}> t(1-\epsilon) x^*\right)  =0. \label{A4}
\eeam
Now from \eqref{A6}, \eqref{A5} and \eqref{A4} we have
\beao
     \limsup_{t\to\infty}b_{I(i)}^{(1)\leftarrow}(t)b_{i-I(i)}^{(2)\leftarrow}(t)\,\P\left(\bZ^{(1)}+\bZ^{(2)}\in tA\right)
        \leq \mu_i^{\oplus}(A_{\epsilon})\downarrow \mu_i^{\oplus}(A) \quad \text{ as }\epsilon\downarrow 0, 
\eeao
where in the last step we use the fact that $\mu_i^{\oplus}(\partial A)=0$.

 \textbf{Step 2.} Next, we derive a lower bound for the asymptotic limit.
There are a total of $2^{|S|}$ subsets of $S$ which we order as
$J(1),\ldots, J(2^{|S|})$ (in any way). Now, define the sets
\beao
    C_{l}:=\bigcap_{j\in J(l)}\left\{Z_j^{(1)}> t x_j\right\}\cap \bigcap_{j\in S\backslash J(l)}\left\{Z_j^{(2)}> t x_j\right\}, \quad
        l=1,\ldots,2^{|S|}.
\eeao
Then,
\beao
    \left\{\bZ^{(1)}+\bZ^{(2)}\in tA\right\}\supseteq\bigcup_{l=1}^{2^{|S|}}C_l,
\eeao
and using the inclusion-exclusion principle we have
\beam \label{A1}
    \P\left(\bZ^{(1)}+\bZ^{(2)}\in tA\right)\geq \P\left(\bigcup_{l=1}^{2^{|S|}}C_l\right)
        \geq \sum_{l=1}^{2^{|S|}}\P(C_l)-\sum_{1\leq l_1<l_2\leq 2^{|S|}}\P(C_{l_1}\cap C_{l_2}).
\eeam
Now on one hand,
\beam \label{A2}
    \lim_{t\to\infty}b_{I(i)}^{(1)\leftarrow}(t)b_{i-I(i)}^{(2)\leftarrow}(t)\sum_{l=1}^{2^{|S|}}\P(C_l)=\mu_i^{\oplus}(A),
\eeam
and on the other hand, for any $1\leq l_1<l_2\leq 2^{|S|}$ the inequality
\beao
    0\leq \P(C_{l_1}\cap C_{l_2})\leq \P(Z_{(|J(l_1)\cup J(l_2)|)}^{(1)}>t x^*)\P(Z_{(|S\backslash (J(l_1)\cap J(l_2))|) }^{(2)}>tx^*)
\eeao
holds.  Define $m:=|J(l_1)\cap J(l_2)|$. Since $J(l_1)\not=J(l_2)$ we have $|J(l_1)\cup J(l_2)|\geq |J(l_1)\cap J(l_2)|+1=m+1$.
Hence, a conclusion of \Cref{Lemma 4.2} is that
\beam \label{A3}
    \lefteqn{\limsup_{t\to\infty}b_{I(i)}^{(1)\leftarrow}(t)b_{i-I(i)}^{(2)\leftarrow}(t)\,\P(C_{l_1}\cap C_{l_2})}  \\
        &&\leq \limsup_{t\to\infty}b_{I(i)}^{(1)\leftarrow}(t)b_{i-I(i)}^{(2)\leftarrow}(t)
         \P(Z_{(m+1)}^{(1)}>t x^*)\P(Z_{(i-m) }^{(2)}>tx^*) =0 \nonumber
\eeam
for any $1\leq l_1<l_2\leq 2^{|S|}$. Then \eqref{A1}, \eqref{A2} and \eqref{A3} result in the lower bound
\beao
    \liminf_{t\to\infty}b_{I(i)}^{(1)\leftarrow}(t)b_{i-I(i)}^{(2)\leftarrow}(t)\,\P\left(\bZ^{(1)}+\bZ^{(2)}\in tA\right)
    \geq \mu_i^{\oplus}(A).
\eeao
and together with the upper bound in Step~1 the lemma is proven.
\end{proof}

\begin{lemma} \label{Lemma B.3}
Let the assumptions of \Cref{mainTheorem} hold and
 $A= \{\bz \in \R_+^d: z_j>x_j \;\forall\; j\in S\}$ be a rectangular set with
 $S\subseteq \mathbb{I}$, $|S|>i$ and  $x_{j}>0$ for $j\in S$ {where $\mu_i^{\oplus}(\partial A)=0$.} Then
 \beao
    \mu_i^{\oplus}(A)=c_i^{I(i)}\mu^{(1)}_{i}(A)
        +c_0^{I(i)}\mu^{(2)}_{i}(A).
 \eeao
\end{lemma}
\begin{proof}
    Since 
    \beao
    \mu_i^{\oplus}(A)=\sum_{m=0}^i c_m^{I(i)}\mu_{m,i}^*(A)
 \eeao
 we have to show that $\sum_{m=1}^{i-1} c_m^{I(i)}\mu_{m,i}^*(A)=0$. But
 \beao
    0&\leq&\sum_{m=1}^{i-1} c_m^{I(i)}\mu_{m,i}^*(A)\\
    &= & \lim_{t\to\infty}
   \sum_{m=1}^{i-1}b_{I(i)}^{(1)\leftarrow}(t)b_{i-{I(i)}}^{(2) \leftarrow}(t)
    \sum_{\genfrac{}{}{0pt}{}{J\subseteq S\cup \{\emptyset\}}{|J|=m}}\P\left(\bigcap_{j\in J}\{Z_j^{(1)}>tx_j\}\right)\P\left(\bigcap_{j\in S\backslash J}\{Z_j^{(2)}>tx_j\}\right)  \\
     &\leq& 2^{i}\lim_{t\to\infty}
    \sum_{m=1}^{i-1}
    b_{I(i)}^{(1)\leftarrow}(t)b_{i-{I(i)}}^{(2)\leftarrow}(t)
    \P\left(Z_{(m)}^{(1)}>t\min_{j\in S}x_j\right)\P\left(Z_{|S|-m}^{(2)}>t\min_{j\in S}x_j\right)\\
    &\leq& 2^{i}
    \sum_{m=1}^{i-1}\lim_{t\to\infty}
    b_{I(i)}^{(1)\leftarrow}(t)b_{i-{I(i)}}^{(2)\leftarrow}(t)
    \P\left(Z_{((m-1)+1)}^{(1)}>t\min_{j\in S}x_j\right)\P\left(Z_{i-(m-1)}^{(2)}>t\min_{j\in S}x_j\right).
 \eeao
 The right hand side is equal to zero due to \Cref{Lemma 4.2}.
\end{proof}

\begin{proof}[Proof of \Cref{mainTheorem}]
Due to \Cref{lem:RiisPi} it is sufficient to study the convergence on the rectangular sets $A= \{\bz \in \R_+^d: z_j>x_j \;\forall\; j\in S\}$ where $S\subseteq \mathbb{I}$, $|S|\ge i$ and $x_j>0$ for $j\in S$ with $\mu_i^{\oplus}(\partial A)=0$. If $|S|=i$, a consequence of \Cref{lem:RiisPi} is that
\beao
    \lim_{t\to\infty}b_{I(i)}^{(1)\leftarrow}(t)b_{i-I(i)}^{(2)\leftarrow}(t) \P\left(\bZ^{(1)}+\bZ^{(2)}\in tA\right)=\mu_i^{\oplus}(A).
\eeao
If $|S|>i$ then $i\leq d-1$. Thus, using \Cref{Lemma 4.2} and similar
{elaborate} calculations as in the proof of \Cref{Lemma 3.4} {(cf. proof of \Cref{Lemma B.3})}
we can show that
\beao
     \lefteqn{\lim_{t\to\infty}b_{I(i)}^{(1)\leftarrow}(t)b_{i-I(i)}^{(2)\leftarrow}(t) \P\left(\bZ^{(1)}+\bZ^{(2)}\in tA\right)} \\
        && =\lim_{t\to\infty}b_{I(i)}^{(1)\leftarrow}(t)b_{i-I(i)}^{(2)\leftarrow}(t) \P\left(\bZ^{(1)}\in tA\right)+
        \lim_{t\to\infty}b_{I(i)}^{(1)\leftarrow}(t)b_{i-I(i)}^{(2)\leftarrow}(t) \P\left(\bZ^{(2)}\in tA\right)\\
        &&=c_i^{I(i)}\mu^{(1)}_{|S|}(A)\lim_{t\to\infty}\frac{b_i^{(1)\leftarrow}(t)}{b_{|S|}^{(1)\leftarrow}(t)}
        +c_0^{I(i)}\mu^{(2)}_{|S|}(A)\lim_{t\to\infty}\frac{b_i^{(2)\leftarrow}(t)}{b_{|S|}^{(2)\leftarrow}(t)}.
\eeao
In  case $\lim_{t\to\infty}\frac{b_i^{(1)\leftarrow}(t)}{b_{|S|}^{(1)\leftarrow}(t)}=0$ we have $\mu^{(1)}_{i}(A)=0$.
Otherwise, $\mu^{(1)}_{|S|}(A)\lim_{t\to\infty}\frac{b_i^{(1)\leftarrow}(t)}{b_{|S|}^{(1)\leftarrow}(t)}=\mu^{(1)}_{i}(A)$.
In summary,
\begin{align*}
     \lim_{t\to\infty}b_{I(i)}^{(1)\leftarrow}(t)b_{i-I(i)}^{(2)\leftarrow}(t) \P\left(\bZ^{(1)}+\bZ^{(2)}\in tA\right)
        & =c_i^{I(i)}\mu^{(1)}_{i}(A)
        +c_0^{I(i)}\mu^{(2)}_{i}(A) =\mu_i^{\oplus}(A)
\end{align*}
where the final equality is due to 
        \Cref{Lemma B.3}.
\end{proof}

\section{Proofs of the results in Section~\ref{subsec:hrvinallcones}} \label{app:subsec:hrvinallcones}

\begin{proof}[Proof of \Cref{prop:ind}]
Note that $\alpha_i=i\alpha$ is immediate from $b_i(t)=(b_1(t))^{1/i} \in \RV_{1/(i\alpha)}$. Using \Cref{mainTheorem}, it is sufficient to prove the statements for $n=2$ and the rest follows by induction (which are direct and not shown here).
Using the notation of \Cref{mainTheorem}, for any $i=1,\ldots, d$, we have $I(i)=i$ and $c_m^{I(i)}=1$ for $m=0,\ldots,i$ for any $i=1,\ldots, d$ and $\lim_{t\to\infty}b_m^{(1)\leftarrow}(t)/b_{m+1}^{(1)\leftarrow}(t)=0.$ Thus,
 $$ \bZ^{(1)} + \bZ^{(2)} \in\MRV(i\alpha,b_1^{1/i},\E_d^{(i)}) $$
 and \eqref{mrv:ind:sum} follows by induction. Now  if  \eqref{star} is satisfied then (for $n=2$), we have $\mu_{m,i}^*(\cdot)=\binom i m \, \mu_i(\cdot)$ and hence, for any $A\in \cB(\E_d^{(i)})$ with $\mu_i(\partial A)=0$ we get
  $$\mu_i^{\oplus}(A)=\sum_{m=0}^{i}\binom i m\, \mu_i(A)=2^{i}\mu_i(A)$$
 implying that
 $$ \bZ^{(1)} + \bZ^{(2)} \in\MRV(i\alpha,b_1^{1/i},2^{i}\mu_i, \E_d^{(i)}).$$
Now \eqref{star} follows by induction using \Cref{mainTheorem}.
\end{proof}

\begin{proof}[Proof of \Cref{prop:mrv:tailsremainsame}]
Consider $n=2$ and the notation of \Cref{mainTheorem}. Fix some $i\in \{1, \ldots, d\}$. We have $I(i)=i$, $c_m^{I(i)}=0, m=1, \ldots, i-1$ and $c_0^{I(i)}=c_i^{I(i)}=1$. Clearly, $\mu^*_{0,i} =\mu^*_{i,i} = \mu_i$ and hence, $\mu^{\oplus}_i = \mu^*_{0,i}+\mu^*_{i,i}= 2\mu_i.$
Now by \Cref{mainTheorem}, we get
\[\bZ^{(1)}+\bZ^{(2)}\in\MRV(\alpha_i,b_i,2\mu_i,\E_d^{(i)}).\]
The final result can now be derived using induction (which we skip here).
\end{proof}

\begin{proof}[Proof of \Cref{cor:asymdep}]
Since $\alpha_i=\alpha<\alpha+\alpha=\alpha_{m}+\alpha_{i-m}$, this holds as a direct consequence of \Cref{prop:mrv:tailsremainsame}.
\end{proof}

\section{Proof of Proposition \ref{prop:Delta1}} \label{sec:proofpropDelta1}

\begin{proof}
By Definition \ref{def:amrv}, $\Delta=1$, $b_1(t)\in \RV_{1/\alpha}$ and for some $0<\gamma<\alpha/d$, $b_i(t)= t^{1/(i(\alpha+\gamma))}, i=2,\ldots, d$. Then $b_1^{\leftarrow}(t)\in \RV_{\alpha}$ and let $b_1^{\leftarrow}(t)=t^{\alpha}\ell(t)$ where $\ell(t)$ is some slowly varying function. Also define $\bS^{(n)}:=\sum_{k=1}^n\bZ^{(k)}$. We prove the statement by induction.
\begin{enumerate}[(1)]
\item Consider $i=1$. By definition, \eqref{eq:mrv:muiplusdelta1} holds for $n=i=1$. Using
classical MRV results \cite[Theorem 1.30]{lindskog:2004thesis}, \citep[Example 3.2]{hult:samorodnitsky:2007}, we obtain $\bS^{(n)}\in \MRV(\alpha,b_1,n\mu_{1},\E_d^{(1)})$. Thus \eqref{eq:mrv:muiplusdelta1} holds for all $n\ge i=1$; the form of the limit measure is \eqref{eq:limit:muiplusdelta1} with $|S|=i=1$.
\item  Fix $i_0\in \{2,\ldots,d\}$. By way of induction, assume that for any $i\in \{1,\ldots,i_0-1\}$, \eqref{eq:mrv:muiplusdelta1} holds for all $n\ge i$ and \eqref{eq:mrv:nciplusdelta1} holds  for $n<i$. First, in part (i) we show that for $i=i_0$, \eqref{eq:mrv:nciplusdelta1}  holds for $n<i=i_0$. Then, we show \eqref{eq:mrv:muiplusdelta1} holds  for $n=i=i_0$ in part (ii) and for all $n> i=i_0$ in part (iii).
Note that the induction base case holds for $i_0=2$. Moreover, for $n=1$, \eqref{eq:mrv:nciplusdelta1} holds for $i=2,\ldots,d$.
\item[(i)] Additionally assume that for $i=i_0$, \eqref{eq:mrv:nciplusdelta1} holds for all $n=1,\ldots, n_0-1$ where $n_0<i_0$. Here we show \eqref{eq:mrv:nciplusdelta1} holds for $i=i_0$ and $n=n_0$. If $i_0=2$, then the only choice of $n$ is  $n=n_0=1$ and \eqref{eq:mrv:nciplusdelta1} holds since $\bS^{(n_0)}=\bZ^{(1)}\in \mathcal{NC}(t^{2(\alpha+\gamma)},\E_d^{(2)})$.  So assume $i_0\ge 3$. Note that $\bS^{(n_0)}=\bS^{(n_0-1)}+\bZ^{(n_0)}$. Using the notation of \Cref{mainTheorem}, for $j=0$,
\begin{align*}
\bar{c}_{0,n_0}^{(i_0)}=\bar{c}_{j,n_0}^{(i_0)}&:=\max_{0\le m\le i_0}\left\{\limsup_{t\to\infty} \frac{b_{j}^{(n_0-1)\leftarrow}(t)b_{i_0-j}^{\leftarrow}(t)}{b_{m}^{(n_0-1)\leftarrow}(t)b_{i_0-m}^{\leftarrow}(t)} \right\} \\
& \ge \limsup_{t\to\infty} \frac{b_{0}^{(n_0-1)\leftarrow}(t)b_{i_0}^{\leftarrow}(t)}{b_{1}^{(n_0-1)\leftarrow}(t)b_{i_0-1}^{\leftarrow}(t)} = \limsup_{t\to\infty} \frac{1\cdot t^{i_0(\alpha+\gamma)}}{t^{\alpha}\ell(t)\cdot t^{(i_0-1)(\alpha+\gamma)}} =\infty,
\end{align*}
 since $ b_{1}^{(n_0-1)}(t)=b_1(t)$ from (1) and $b_1^{\leftarrow}(t)=t^{\alpha}\ell(t)$. Similarly for $j=1$,
\begin{align*}
\bar{c}_{1,n_0}^{(i_0)}& =\max_{0\le m\le i_0}\left\{\limsup_{t\to\infty} \frac{b_{1}^{(n_0-1)\leftarrow}(t)b_{i_0-1}^{\leftarrow}(t)}{b_{m}^{(n_0-1)\leftarrow}(t)b_{i_0-m}^{\leftarrow}(t)} \right\} \\
& =\max\left\{0, 1, \max_{2\le m \le i_0-1}\limsup_{t\to\infty} \frac{t^{\alpha}\ell(t) \cdot t^{(i_0-1)(\alpha+\gamma)}}{b_{m}^{(n_0-1)\leftarrow}(t)\cdot t^{(i_0-m)(\alpha+\gamma)}}, \limsup_{t\to\infty} \frac{t^{\alpha}\ell(t) \cdot t^{(i_0-1)(\alpha+\gamma)}}{b_{i_0}^{(n_0-1)\leftarrow}(t)b_{0}^{(1)\leftarrow}(t)} \right\} \\
& = \max\left\{0, 1,  \limsup_{t\to\infty} \frac{t^{i_0\alpha+(i_0-1)\gamma}\ell(t)}{t^{i_0(\alpha+\gamma)}}, \limsup_{t\to\infty} \frac{t^{i_0\alpha+(i_0-1)\gamma}\ell(t)}{t^{i_0(\alpha+\gamma)}}  \right\} =1.
\end{align*}
The third equality results from $b_m^{(n_0-1)\leftarrow}(t)= t^{m(\alpha+\gamma)}$ for $m\in \{2,\ldots, i_0-1\}$ where $n_0-1< i_0-1$  and $b_{i_0}^{(n_0-1)\leftarrow}(t) =t^{i_0(\alpha+\gamma)} $ (by induction assumption). Similarly, for $j=i_0-1$, we can check that
\[c_{i_0-1,n_0}=1.\] Finally, for $j\in \{2,\ldots, i_0-2, i_0\}$,
\begin{align*}
\bar{c}_{j,n_0}^{(i_0)}& =\max_{0\le m\le i_0}\left\{\limsup_{t\to\infty} \frac{b_{j}^{(n_0-1)\leftarrow}(t)b_{i_0-j}^{\leftarrow}(t)}{b_{m}^{(n_0-1)\leftarrow}(t)b_{i_0-m}^{\leftarrow}(t)} \right\} \\
& \ge \limsup_{t\to\infty} \frac{b_{j}^{(n_0-1)\leftarrow}(t)b_{i_0-j}^{\leftarrow}(t)}{b_{1}^{(n_0-1)\leftarrow}(t)b_{i_0-1}^{\leftarrow}(t)} = \limsup_{t\to\infty} \frac{t^{j(\alpha+\gamma)}\cdot t^{(i_0-j)(\alpha+\gamma)}}{t^{\alpha}\ell(t)\cdot t^{(i_0-1)(\alpha+\gamma)}} =\infty.
\end{align*}
Hence, $I(i_0)=1$ with
\begin{align*}
c_{m,n_0}^{I(i_0)} =\begin{cases}
    1, \quad m \in \{1,i_0-1\},\\
    0, \quad \text{otherwise}.
\end{cases}
\end{align*}
Therefore by \Cref{mainTheorem}, $\bS^{(n_0)}\in\MRV^*(\alpha_{i_0,n_0},\widetilde{b}_{i_0,n_0},\mu_{i_0,n_0}^{\oplus},\E_d^{(i_0)})$ where $\alpha_{i_0,n_0}=\alpha_{1,n_0-1}+\alpha_{i_0-1}=1+\infty=\infty$ (recall $i_0\ge 3$), $\widetilde{b}_{i_0,n_0}(t)=t^{i_0\alpha+(i_0-1)\gamma}\ell(t)$ and $\mu_{i_0,n_0}^{\oplus}\equiv 0$. 
Defining ${b}_{i_0,n_0}(t) = t^{i_0(\alpha+\gamma)}=b_{i_0}(t)$, since  $\widetilde{b}_{i_0}(t)=o(b_{i_0}(t))$, we have $\bS^{(n_0)}\in\mathcal{NC}({b}_{i_0}(t),\E_d^{(i_0)})$. 
Therefore, by induction, \eqref{eq:mrv:nciplusdelta1} holds for all $n<i_0$.

\item[(ii)] Now, we show \eqref{eq:mrv:muiplusdelta1} holds for $n= n_0= i_0$.  For $j=0,\ldots,i_0-2$,
\begin{align*}
\bar{c}_{j,n_0}^{(i_0)}=\bar{c}_{j,i_0}^{(i_0)}&:=\max_{0\le m\le i_0}\left\{\limsup_{t\to\infty} \frac{b_{j}^{(i_0-1)\leftarrow}(t)b_{i_0-j}^{\leftarrow}(t)}{b_{m}^{(i_0-1)\leftarrow}(t)b_{i_0-m}^{\leftarrow}(t)} \right\} \\
& \ge \limsup_{t\to\infty} \frac{b_{j}^{(i_0-1)\leftarrow}(t)b_{i_0-j}^{\leftarrow}(t)}{b_{i_0-1}^{(i_0-1)\leftarrow}(t)b_{1}^{\leftarrow}(t)} \\ & = \limsup_{t\to\infty} \frac{(t^\alpha\ell(t))^j\cdot t^{(i_0-j)(\alpha+\gamma)}}{(t^{\alpha}\ell(t))^{(i_0-1)}\cdot t^{\alpha}\ell(t)} =\infty,
\end{align*}
since by induction assumption  $b_j^{(n_0-1)}(t)= b_j^{(i_0-1)}(t)=(b_1(t))^{1/j}$ for all $j\le n_0-1$. Similarly for $j=i_0$, we have
$$\bar{c}_{i_0,n_0}=\bar{c}_{i_0,i_0}=\infty$$ since we have $b_{i_0}^{(i_0-1)}(t)=t^{i_0(\alpha+\gamma)}$ from part 2(i) of the proof. For $j=i_0-1$,
 \begin{align*}
\bar{c}_{i_0-1,n_0}^{(i_0)} = \bar{c}_{i_0-1,i_0}^{(i_0)}&:=\max_{0\le m\le i_0}\left\{\limsup_{t\to\infty} \frac{b_{i_0-1}^{(i_0-1)\leftarrow}(t)b_{1}^{\leftarrow}(t)}{b_{m}^{(i_0-1)\leftarrow}(t)b_{i_0-m}^{\leftarrow}(t)} \right\} \\
& = \max_{0\le m\le i_0} \left\{\limsup_{t\to\infty} \frac{(t^{\alpha}\ell(t))^{(i_0-1)}\cdot t^{\alpha}\ell(t)}{b_{m}^{(i_0-1)\leftarrow}(t)b_{i_0-m}^{\leftarrow}(t)}\right\} = 1.
\end{align*}
Finally, for $j=i_0$,
 \begin{align*}
\bar{c}_{i_0,n_0}^{(i_0)} = \bar{c}_{i_0,i_0}^{(i_0)}&:=\max_{0\le m\le i_0}\left\{\limsup_{t\to\infty} \frac{b_{i_0}^{(i_0-1)\leftarrow}(t)b_{0}^{\leftarrow}(t)}{b_{m}^{(i_0-1)\leftarrow}(t)b_{i_0-m}^{\leftarrow}(t)} \right\} \\
&  \ge \limsup_{t\to\infty} \frac{t^{i_0(\alpha+\gamma)}}{b_{i_0-1}^{(i_0-1)\leftarrow}(t)b_{1}^{\leftarrow}(t)} = \infty.
\end{align*}
Hence, $I(i_0)=i_0-1$ with
\begin{align*}
c_{m,n_0}^{I(i_0)}= c_{m,i_0}^{I(i_0)} =\begin{cases}
    1, \quad m=i_0-1,\\
    0, \quad \text{otherwise}.
\end{cases}
\end{align*}
Therefore, by \Cref{mainTheorem}, $\bS^{(n_0)}=\bS^{(i_0)}\in\MRV(\alpha_{i_0,n_0},b_{i_0,n_0},\mu_{i_0,n_0}^{\oplus},\E_d^{(i_0)})$ where $$\alpha_{i_0,n_0}=\alpha_{i_0,i_0}=\alpha_{i_0-1,i_0-1}+\alpha_{1}=(i_0-1)\alpha+\alpha=i_0\alpha,$$ ${b}_{i_0,n_0}^{\leftarrow}(t)= b_{i_0}^{(i_0-1)\leftarrow}(t)b_1^{\leftarrow}(t)=(t^\alpha\ell(t))^{i_0}$ and
\begin{align*}
\mu_{i_0,n_0}^{\oplus} &= \mu_{i_0,i_0}^{\oplus} = \sum_{m=0}^{i_0} c_{m,i_0}^{I(i_0)}\mu^*_{m,i_0,i_0} = \mu_{i_0-1,i_0,i_0},
\end{align*}
where for $A=\{\bz\in\R_+^{d}:z_j>x_j \;\forall\; j\in S\}\in \mathcal{R}^{(i_0)}$ with $|S|= i_0$, $x_j>0$ for $j\in S$
\begin{align*}
& \mu_{i_0-1,i_0,i_0}(A)\\
& = \sum_{\genfrac{}{}{0pt}{}{J\subseteq S}{|J|=i_0-1}} \mu_{i_0-1,i_0-1}^{\oplus} \,\Big(\{\bz\in\E_d^{(i_0-1)}:  z_{j}>x_{j}\,\forall\,j\in J\}\}\Big)  \mu_{1}\left(\{\bz\in\E_d^{(1)}:z_{j}>x_{j} \;\forall\, j\in S\setminus J\}\right)\\
& = \sum_{\genfrac{}{}{0pt}{}{J\subseteq S}{|J|=i_0-1}} \left[(i_0-1)!\prod_{j\in J}
    \mu_1\left(\{\bz\in \E_d^{(i)}:z_j>x_j\}\right)\right] \mu_1\left(\{\bz\in \E_d^{(i)}:z_j>x_j\; \forall\; j\in S\setminus J\}\right) \\
&= \sum_{j\in S}  (i_0-1)!\left[ \prod_{k\in S\setminus \{j\}}
    \mu_1\left(\{\bz\in \E_d^{(i)}:z_k>x_k\}\right)\right]
    \mu_1\left(\{\bz\in \E_d^{(i)}:z_j>x_j\}\right)\\
    &= i_0! \prod_{j\in S}
    \mu_1\left(\{\bz\in \E_d^{(i)}:z_j>x_j\}\right).
\end{align*}
Hence, \eqref{eq:mrv:muiplusdelta1} holds for $n_0=i_0$.
\item[(iii)] Here we show \eqref{eq:mrv:muiplusdelta1} holds for all $n\ge i_0$. By way of induction (additionally) assume that for $i=i_0$, \eqref{eq:mrv:muiplusdelta1} holds for all $n\in \{i_0,i_0+1,\ldots,n_0\}$. We will show that then it also holds for $n=n_0+1$. By part 2(ii), we know that it holds for $n_0=i_0$. For $j=0, \ldots, i_0-2$,
\begin{align*}
\bar{c}_{j,n_0+1}^{(i_0)}&:=\max_{0\le m\le i_0}\left\{\limsup_{t\to\infty} \frac{b_{j}^{(n_0)\leftarrow}(t)b_{i_0-j}^{\leftarrow}(t)}{b_{m}^{(n_0)\leftarrow}(t)b_{i_0-m}^{\leftarrow}(t)} \right\} \\
& \ge \limsup_{t\to\infty} \frac{b_{j}^{(n_0)\leftarrow}(t)b_{i_0-j}^{\leftarrow}(t)}{b_{i_0}^{(n_0)\leftarrow}(t)b_{0}^{\leftarrow}(t)} \\ & = \limsup_{t\to\infty} \frac{(t^{\alpha}\ell(t))^{j}\cdot t^{(i_0-j)(\alpha+\gamma)}}{(t^{\alpha}\ell(t))^{i_0}\cdot 1} = \limsup_{t\to\infty} t^{(i_0-j)\gamma}(\ell(t))^{(j-i_0)} = \infty,
\end{align*}
since by induction assumption  $b_j^{(n_0)}(t) =(b_1(t))^{1/j}$ for all $j\le n_0$. By similar arguments we have for
\begin{align*}
\bar{c}_{i_0-1,n_0+1}^{(i_0)} =\bar{c}_{i_0,n_0+1}^{(i_0)}=1.
\end{align*}
Hence, $I(i_0)=i_0-1$, and
\begin{align*}
c_{m,n_0+1}^{I(i_0)} =\begin{cases}
    1, \quad m \in \{i_0-1, i_0\}, \\
    0, \quad \text{otherwise}.
\end{cases}
\end{align*}
Therefore by \Cref{mainTheorem}, $\bS^{(n_0+1)}\in\MRV(\alpha_{i_0,n_0+1},b_{i_0,n_0+1},\mu_{i_0,n_0+1}^{\oplus},\E_d^{(i_0)})$ where 
\begin{align*}
\alpha_{i_0,n_0+1} & =\alpha_{i_0-1,n_0}+\alpha_{1}=(i_0-1)\alpha+\alpha=i_0\alpha,\\
{b}_{i_0,n_0+1}^{\leftarrow}(t) & = b_{i_0-1}^{(n_0)\leftarrow}(t)b_1^{\leftarrow}(t)=(t^\alpha\ell(t))^{i_0}, \text{ and},\\
\mu_{i_0,n_0+1}^{\oplus} &= \sum_{m=0}^{i_0} c_{m,n_0+1}^{I(i_0)}\mu^*_{m,i_0,n_0+1} = \mu_{i_0-1,i_0,n_0+1} +  \mu_{i_0,i_0,n_0+1},
\end{align*}
where for $A=\{\bz\in\R_+^{d}:z_j>x_j\;\forall\; j\in S\}\in \mathcal{R}^{(i_0)}$ with $|S|= i_0$, $x_j>0$ for  $j\in S$
\begin{align*}
& \mu_{i_0-1,i_0,n_0+1}(A)\\
& = \sum_{\genfrac{}{}{0pt}{}{J\subseteq S}{|J|=i_0-1}} \mu_{i_0-1,n_0}^{\oplus} \,\Big(\{\bz\in\E_d^{(i_0-1)}: z_{j}>x_{j} \;\forall\;j\in J\}\Big)  \mu_{1}\left(\{\bz\in\E_d^{(1)}:z_{j}>x_{j}\;\forall\;j\in S\setminus J\}\right)\\
& = \sum_{\genfrac{}{}{0pt}{}{J\subseteq S}{|J|=i_0-1}} \left[\frac{n_0!}{(n_0-i_0+1)!}\prod_{j\in J}
    \mu_1\left(\{\bz\in \E_d^{(i)}:z_j>x_j\}\right)\right] \mu_1\left(\{\bz\in \E_d^{(i)}:z_j>x_j, j\in S\setminus J\}\right) \\
&= \sum_{j\in S}  \frac{n_0!}{(n_0-i_0+1)!}\left[ \prod_{k\in S\setminus \{j\}}
    \mu_1\left(\{\bz\in \E_d^{(i)}:z_k>x_k\}\right)\right]
    \mu_1\left(\{\bz\in \E_d^{(i)}:z_j>x_j\}\right)\\
    &= i_0 \cdot \frac{n_0!}{(n_0-i_0+1)!} \prod_{j\in S}
    \mu_1\left(\{\bz\in \E_d^{(i)}:z_j>x_j\}\right),
\end{align*}
and,
\begin{align*}
 \mu_{i_0,i_0,n_0+1}(A) & = \sum_{\genfrac{}{}{0pt}{}{J\subseteq S}{|J|=i_0}} \mu_{i_0,n_0}^{\oplus} \,\Big(\{\bz\in\E_d^{(i_0-1)}: z_{j}>x_{j} \;\forall\;j\in J\}\Big) \\
& = \frac{n_0!}{(n_0-i_0)!}\prod_{j\in J}
    \mu_1\left(\{\bz\in \E_d^{(i_0)}:z_j>x_j\}\right).
\end{align*}
The measures $\mu_{i_0-1,n_0}^{\oplus},\mu_{i_0,n_0}^{\oplus}, \mu_1 $ are obtained from our assumptions and induction hypothesis. Now,
\begin{align*}
\mu_{i_0,n_0+1}^{\oplus}(A)& = \mu_{i_0-1,i_0,n_0+1}(A)+\mu_{i_0,i_0,n_0+1}(A)\\
   & = \left[ i_0 \cdot \frac{n_0!}{(n_0-i_0+1)!} + \frac{n_0!}{(n_0-i_0)!} \right] \prod_{j\in S}
    \mu_1\left(\{\bz\in \E_d^{(i)}:z_j>x_j\}\right)\\
    & = \frac{(n_0+1)!}{(n_0+1-i_0)!} \prod_{j\in S}
    \mu_1\left(\{\bz\in \E_d^{(i)}:z_j>x_j\}\right).
\end{align*}
Hence, \eqref{eq:mrv:muiplusdelta1} holds for $i=i_0$ and $n=n_0+1$, thus by induction it holds for all $n\ge i_0$.
\end{enumerate}
\end{proof}
\section{Proof of  Theorem \ref{thm:randomsum:main}} \label{sec:proofmainrandomsum}

For the proof we require some auxiliary results.
\begin{lemma} \label{aux_1}
Let the assumptions of \Cref{thm:randomsum:main} hold. Define $\bZ^{\oplus}:=(Z^{\oplus}_1,\ldots, Z^{\oplus}_d):=\sum_{k=1}^{d}\bZ^{(k)}$
and denote by  
$Z^{\oplus}_{(1)}\geq \ldots \geq Z^{\oplus}_{(d)}$
 the order statistics of
$Z^{\oplus}_1,\ldots, Z^{\oplus}_d$. Also {let $Z_{(1)}^{(k)}\geq\ldots Z_{(d)}^{(k)})$ be the order statistics of the elements of $\bZ^{(k)}=(Z_{1}^{(k)},\ldots, Z_{d}^{(k)})$ for any $k\ge 1$.} Furthermore,
 for $n\in\N$ and $i=1,\ldots,d$ define
\beao
    \alpha_{i,n}:=\sup_{S\subseteq \mathbb{I}, \,
        |S|\leq  i }\sup_{t>0}\frac{\P\left(\bigcap_{j\in S}\left\{\sum_{k=1}^nZ^{(k)}_j>t\right\}\right)}{\P\left(Z^{\oplus}_{(|S|)}>t\right)}.
\eeao
Then  there exists a finite constant $K_i>0$
such that for any $n\in \N$:
\beao
    \alpha_{i,n+1}\leq  K_i^{n}.
\eeao
\end{lemma}
For one-dimensional random variables with $i=d=1$, a stronger result holds: for any $\epsilon>0$ there exists a constant $K>0$ such that the left hand side is bounded by $K(1+\epsilon)^n$ (cf. \cite[Lemma 1.3.5]{embrechts:kluppelberg:mikosch:1997}).

\begin{proof}
    First, we show recursively that for any $i\in\mathbb{I}$ there exists a constant $K_i>0$
    such that $\alpha_{i,n+1}\leq K_i\alpha_{i,n}$ for any $n\geq d$.
    Let $S\subseteq \mathbb{I}$ with $|S|\leq i$ and $n\geq d$. Then
    \beam \label{E1}
        \lefteqn{\P\left(\bigcap_{j\in S}\left\{\sum_{k=1}^{n+1} Z^{(k)}_j>t\right\}\right)} \nonumber\\
            &=&\sum_{\genfrac{}{}{0pt}{}{J\subseteq S}{J\not=\emptyset} }
            \P\left(\bigcap_{j\in S}\left\{\sum_{k=1}^{n+1} Z^{(k)}_j>t\right\}\cap
                \bigcap_{j\in J}\left\{ Z^{(n+1)}_j>t/2\right\}\cap
                \bigcap_{j\in S\backslash J}\left\{ Z^{(n+1)}_j\leq t/2\right\}\right)\nonumber \\
            &&+\P\left(\bigcap_{j\in S}\left\{\sum_{k=1}^{n+1} Z^{(k)}_j>t\right\}\cap
                \bigcap_{j\in S}\left\{ Z^{(n+1)}_j\leq t/2\right\}\right) \nonumber\\
            &=:&J_{n,1}(t,S)+J_{n,2}(t,S).
            \eeam
    We investigate the two terms separately. First,
    \beao
        \lefteqn{J_{n,1}(t,S)}\\
            &&\leq \sum_{\genfrac{}{}{0pt}{}{J\subseteq S}{J\not=\emptyset} }
            \P\left(\bigcap_{j\in S\backslash J}\left\{\sum_{k=1}^{n+1} Z^{(k)}_j>t\right\}\cap
                \bigcap_{j\in J}\left\{ Z^{(n+1)}_j>t/2\right\}\right)\\
            &&\leq \sum_{\genfrac{}{}{0pt}{}{J\subseteq S}{J\not=\emptyset} }\sum_{K\subseteq S\backslash J \cup\{\emptyset\}}
            \P\left(\bigcap_{j\in S\backslash (J\cup K) }\left\{\sum_{k=1}^{n} Z^{(k)}_j>t/2\right\}\cap
                \bigcap_{j\in J}\left\{ Z^{(n+1)}_j>t/2\right\}\cap \bigcap_{j\in K}\left\{ Z^{(n+1)}_j>t/2\right\}\right)\\
            &&\leq \sum_{\genfrac{}{}{0pt}{}{J\subseteq S}{J\not=\emptyset} }\sum_{K\subseteq S\backslash J\cup\{\emptyset\}}
            \P\left(\bigcap_{j\in S\backslash (J\cup K)}\left\{\sum_{k=1}^{n} Z^{(k)}_j>t/2\right\}\right)
                \P\left(\bigcap_{j\in J\cup K}\left\{ Z^{(n+1)}_j>t/2\right\}\right).
    \eeao
    Since the set $S\backslash J\cup K$ has at most $i-1$ elements and by definition $\alpha_{i-1,n}\leq \alpha_{i,n}$, we have that

    \beao
        J_{n,1}(t,S)
            \leq \alpha_{i,n}\sum_{\genfrac{}{}{0pt}{}{J\subseteq S}{J\not=\emptyset} }\sum_{K\subseteq S\backslash J\cup\{\emptyset\}}
            \P\left( Z_{(|S\backslash (J\cup K)|)}^{\oplus}>t/2\right)
                \P\left(Z_{(|J\cup K|)}^{(n+1)}>t/2\right).
    \eeao
    Now applying  \eqref{E6} and \eqref{E7} we have
    \beam \label{E2}
        \sup_{\genfrac{}{}{0pt}{}{S\subseteq \mathbb{I}}{ |S|\leq i }}\sup_{t>0}\frac{J_{n,1}(t,S)}{\P\left(Z_{(|S|)}^{\oplus}>t\right)}&\leq& \alpha_{i,n}\sup_{\genfrac{}{}{0pt}{}{S\subseteq \mathbb{I}}{ |S|\leq i }}\sum_{\genfrac{}{}{0pt}{}{J\subseteq S}{J\not=\emptyset} }\sum_{K\subseteq S\backslash J\cup\{\emptyset\}}C^*C^{**} \nonumber\\
        &\leq & \alpha_{i,n} 2^{2i} C^*C^{**} =\alpha_{i,n}\widetilde K_i
    \eeam
    with $\widetilde K_i:=2^{2i}C^*C^{**} $.
    Next, for the second term in \eqref{E1} and $S=\{j_1,\ldots,j_i\}$ we have
    \beao
        \lefteqn{\sup_{\genfrac{}{}{0pt}{}{S\subseteq \mathbb{I}}{ |S|\leq i }}\sup_{t>0}\frac{J_{n,2}(t,S)}{\P\left(Z_{|S|}^{\oplus}>t\right)}} \nonumber\\
            &&=\sup_{\genfrac{}{}{0pt}{}{S\subseteq \mathbb{I}}{ |S|\leq i }}\sup_{t>0}\int_0^{t/2}\cdots\int_0^{t/2}
            \frac{\P\left(\bigcap_{j\in S}\left\{\sum_{k=1}^{n} Z^{(k)}_j>t-y_j\right\}\right)}{\P\left(Z_{|S|}^{\oplus}>t/2\right)}\nonumber\\
            &&\quad \quad\qquad\qquad \cdot\frac{\P\left(Z_{|S|}^{\oplus}>t/2\right)}{\P\left(Z_{|S|}^{\oplus}>t\right)}F_{Z_{j_1},\ldots,Z_{j_i}}(dy_1,\ldots,dy_i) \nonumber\\
            &&\leq C^{**}\sup_{\genfrac{}{}{0pt}{}{S\subseteq \mathbb{I}}{ |S|\leq i }}\sup_{t>0}\int_0^{t/2}\cdots\int_0^{t/2}
            \frac{\P\left(\bigcap_{j\in S}\left\{ \sum_{k=1}^{n} Z^{(k)}_j>t/2\right\}\right)}{\P\left(Z_{(|S|)}^{\oplus}>t/2\right)}F_{Z_{j_1},\ldots,Z_{j_i}}(dy_1,\ldots,dy_i)            \nonumber
    \eeao
    where we applied \eqref{E7} once more. Now the last term above is bounded by $C^{**} \alpha_{i,n}$ and hence we have
    \beam \label{E3}
        \sup_{\genfrac{}{}{0pt}{}{S\subseteq \mathbb{I}}{ |S|\leq i} }\sup_{t>0}\frac{J_{n,2}(t,S)}{\P\left(Z_{|S|}^{\oplus}>t\right)}\leq C^{**} \alpha_{i,n}.
    \eeam 
    Now from
    \eqref{E1}, \eqref{E2} and \eqref{E3} we get
    \beam \label{E4}
        \alpha_{i,n+1}\leq \alpha_{i,n}\widetilde K_i+\alpha_{i,n} C^{**}= (\widetilde K_i+C^{**})\alpha_{i,n}.
    \eeam
    Note that  $\alpha_{i,d}\leq 1$ for $i=1,\ldots,d$. Thus applying  \eqref{E4} recursively we obtain \linebreak $\alpha_{i,n+1}\leq (\widetilde K_i+C^{**})^{{n-d}}$ for $n\geq d$, $i=1,\ldots,d$. But for $n\leq d$
we have of course $\alpha_{i,n}\leq 1$ for $i=1,\ldots,d$. Thus, with $K_i=\max(1,\widetilde K_i+C^{**})$  the statement of the lemma is satisfied.
\end{proof}

\begin{proof}[Proof of \Cref{thm:randomsum:main}]
Define $\bZ^{\oplus}:=(Z^{\oplus}_1,\ldots, Z^{\oplus}_d):=\sum_{k=1}^{d}\bZ^{(k)}$ and denote by  
$Z^{\oplus}_{(1)}\geq \ldots \geq Z^{\oplus}_{(d)}$
 the order statistics of
$Z^{\oplus}_1,\ldots, Z^{\oplus}_d$.
Let  $A= \{\bz \in \R_+^d: z_j>x_j \,\forall\; j\in S\}$ be a rectangular set in $\E_d^{(i)}$  where $S\subseteq \mathbb{I}$, $|S|\geq i$ and $x_j>0,\; \forall\, j\in S$ with $\mu_i(\partial A)=0$. Suppose $\widetilde S\subseteq S$ with $|\widetilde S|=i$. Then
\beam \label{E5}
    \lefteqn{\lim_{t\to\infty}\frac{\P\left(\sum_{k=1}^{\tau}\bZ^{(k)}\in tA\right)}{\P\left(Z_{(i)}^{\oplus}>t\right)}} \nonumber\\
        &&=\lim_{t\to\infty}\sum_{n=0}^{\infty}\P(\tau=n)
        \frac{\P\left(\bigcap_{j\in S}\left\{\sum_{k=1}^nZ_{j}^{(k)}>tx_{j}\right\}\right)}{
        \P\left(Z_{(i)}^{\oplus}>t\right)}.
\eeam
But for any $n\in\N$ we have
\beam  \label{E11}
    0&\leq& \sup_{t>0}\frac{\P\left(\bigcap_{j\in S}\left\{\sum_{k=1}^nZ_{j}^{(k)}>tx_{j}\right\}\right)}{
        \P\left(Z_{(i)}^{\oplus}>t\right)} \nonumber\\
    &\leq& \sup_{t>0}\frac{\P\left(\bigcap_{j\in \widetilde S}\left\{\sum_{k=1}^nZ_{j}^{(k)}>t\min_{j\in S}x_{j}\right\}\right)}{
        \P\left(Z_{(i)}^{\oplus}>t\min_{j\in S} x_{j}\right)}
        \frac{        \P\left(Z_{(i)}^{\oplus}>t\min_{j\in S} x_{j}\right)}{
        \P\left(Z_{(i)}^{\oplus}>t\right)}\nonumber \\
    &\leq& \alpha_{i,n} \sup_{t>0}\frac{\P\left(Z_{(i)}^{\oplus}>t\min_{j\in S} x_{j}\right)}{
        \P\left(Z_{(i)}^{\oplus}>t\right)}.
\eeam
Since $\bZ^{(\oplus)}\in\MRV(\alpha_i,b_i,f_i(d)\mu_i,\E_d^{(i)})$ and $f_i(d)\mu_i\left(\{\bz\in\R_+^d:z_{(i)}> 1\}\right)>0$, we have
\beao
    1\leq  \frac{\P\left(Z_{(i)}^{\oplus}>t\min_{j\in S} x_{j}\right)}{
        \P\left(Z_{(i)}^{\oplus}>t\right)}\stackrel{t\to\infty}{\to}
        \left(\min_{j\in S} x_{j}\right)^{-\alpha_i}<\infty.
\eeao
Hence, there exists a finite constant $C>0$ such that
\beam \label{4.8}
    \sup_{t>0}\frac{\P\left(Z_{(i)}^{\oplus}>t\min_{j\in S} x_{j}\right)}{
        \P\left(Z_{(i)}^{\oplus}>t\right)}\leq C.
\eeam
Then an application of \Cref{aux_1} and \eqref{E11}, \eqref{4.8} yield
\beao
    0\leq \sup_{t>0}\frac{\P\left(\bigcap_{j\in S}\left\{\sum_{k=1}^nZ_{j}^{(k)}>tx_{j}\right\}\right)}{
        \P\left(Z_{(i)}^{\oplus}>t\right)}\leq  CK_i^n, \quad n\in\N.
\eeao
Thus, there exists a uniform finite upper bound of the right hand side of \eqref{E5}
such that due to Pratt's Theorem
we are allowed to exchange the limit and the sum. A conclusion of Assumption~\ref{AssumptionA}
is then
\beao
    \lim_{t\to\infty}\frac{\P\left(\sum_{k=1}^{\tau}\bZ^{(k)}\in tA\right)}{\P\left(Z_{(i)}^{\oplus}>t\right)}
    =\sum_{n=0}^{\infty}\P(\tau=n)f_i(n)\frac{\mu_i(A)}{f_i(d)\mu_i\left(\{\bz\in\R_+^d:z_{(i)}> 1\}\right)}.
\eeao
Then \Cref{prop:rectsetsforM} and $\bZ^{\oplus}\in\MRV(\alpha_i,b_i,f_i(d)\mu_i,\E_d^{(i)})$ for $i=1,\ldots,d$ result in $\sum_{k=1}^{\tau}\bZ^{(k)}\in \MRV(\alpha_i,b_i,\E(f_i(\tau))\mu_i,\E_d^{(i)})$.
\end{proof}
\newpage
\section{Proofs of the results in Section \ref{section:Levy}} \label{sec:proofLevy}

\begin{proof}[Proof of \Cref{Proposition 6.1}]  \mbox{}\\
\textbf{Step 1.} To begin with, let $(\bL(s))_{s\geq 0}$ be a compound Poisson process with intensity $\lambda>0$ and jump size distribution $\P_{\bZ}=\Pi/\lambda$, which is a proper probability measure on $\R^d_+$. Let us also assume that $(N(s))_{s\geq 0}$ is a Poisson process with intensity $\lambda$ and $\bZ^{(1)},\bZ^{(2)},\ldots$ are  i.i.d.  with distribution $\P_{\bZ}$.
Then $\P_{\bZ}\in\MRV(\alpha_i,b_i,\mu_i/\lambda,\E_d^{(i)})$.
Since $\E(N(s))=\lambda s$, using \Cref{prop:mrv:tailsremainsame} and \Cref{thm:randomsum:main} we have
\beao
  \bL(s)\eqd\sum_{k=1}^{N(s)}\bZ^{(k)}\in\MRV(\alpha_i,b_i,s\mu_i,\E_d^{(i)}) \quad  \text {for } \quad i=1,\ldots,d.
\eeao
\textbf{Step 2.} Now let $(\bL(s))_{s\geq 0}$ be a general L\'evy process.
Define $D_{a,\infty}:=\{\bz\in\R^d:a<\|\bz\|<\infty\}$ for any $a>0$.
Due to the L\'evy-Itô decomposition (see \cite[Theorem 19.2 and Theorem 19.3]{sato:1991}) we can decompose $\bL$ into two independent L\'evy processes $\bL_1=(\bL_1(s))_{s\geq 0}$ and $\bL_2=(\bL_2(s))_{s\geq 0}$ such that
$$\bL(s)=\bL_1(s)+\bL_2(s), \quad s\geq 0,$$
where $\bL_1$ is a compound Poisson process with L\'evy measure $\Pi(\cdot\cap D_{a,\infty})/\Pi(D_{a,\infty})$\ and Poisson intensity
$\Pi(D_{a,\infty})$, whereas  $\bL_2$ satisfies  $\E\|\bL_2(s)\|^\theta<\infty$ for any $\theta>0$ (see \cite[Lemma 2.2 and proof of Theorem 2.3]{lindskog:2004thesis}). Thus, the L\'evy measure of $\bL_1$ is $\Pi(\cdot\cap D_{a,\infty})\in \MRV(\alpha_i,b_i,\mu_i,\E_d^{(i)})$ for $i=1,\ldots,d$
and by step 1 we have
\beao
  \bL_1(s)\in\MRV(\alpha_i,b_i,s\mu_i,\E_d^{(i)}) \quad  \text {for } \quad i=1,\ldots,d.
\eeao
Then an application of \Cref{Lemma_aux} and $\bL(s)=\bL_1(s)+\bL_2(s)$ gives us the result.
\end{proof}

\begin{proof}[Proof of \Cref{Proposition 6.3}]
As in \Cref{Proposition 6.1} it is sufficient to investigate compound Poisson processes $(\bL(s))_{s\geq 0}=(\sum_{k=1}^{N(s)}\bZ^{(k)})_{s\geq 0}$  with intensity $\lambda>0$ and jumps size distribution $\P_{\bZ}=\Pi/\lambda$. Then
for the jump size distribution we have $\P_{\bZ}=\Pi/\lambda\in\MRV(\alpha_i,b_i,\mu_i/\lambda,\E_d^{(i)})$ for $i=1,\ldots,d$.
Since $\E(N(s)(N(s)-1)(N(s)-i+1))=(\lambda s)^{i}$, $f_i(n)=0$ for $n<i$, $f_i(n)=n!/(n-i)!$ for $n\geq i$, \Cref{prop:Delta1}  and \Cref{thm:randomsum:main} result in
\begin{eqnarray*}
    \bL(s)=\sum_{k=1}^{N(s)}\bZ^{(k)}\in \MRV(i\alpha_1,b_1^{1/i},s^{i}\mu_i^L,\E_d^{(i)})
    \quad \text {for } \quad i=1,\ldots,  d,
\end{eqnarray*}
which is the statement.
\end{proof}

\begin{proof}[Proof of \Cref{Corollary:2}]  Suppose $( \bL(s))_{s\geq 0}$ is a compound Poisson process with $\Pi(\R^d_+)\leq 1$. Let $\bZ^{(1)},\bZ^{(2)},\ldots$ be a sequence of  i.i.d.  random vectors with distribution $\P(\bZ^{(1)}=\bzero)=1-\Pi(\R_+^d)$ and
$\P(\bZ^{(1)}\in A\backslash\{\bzero\})=\Pi(A\backslash\{\bzero\})$
for all sets $A\in \mathcal{B}(\R_+^d)$. Then
$\bZ^{(1)}\in\MRV(\alpha_i,b_i,\mu_i,\E_d^{(i)})$  for $i=1,\ldots,d$. Due to \Cref{prop:ind} and \Cref{thm:randomsum:main} we receive
\begin{eqnarray*}
    \bL(s)\eqd \sum_{k=1}^{N^*(s)}\bZ^{(k)}\in\MRV(i\alpha,b_1^{1/i},\E(N^*(s)^{i})\mu_i,\E_d^{(i)}).
\end{eqnarray*}
We extend this result to general L\'evy measures and L\'evy processes as in \Cref{Proposition 6.1} by choosing $a$ large enough so that $\Pi(D_{a,\infty})\leq 1$.
\end{proof}

\end{document}